\documentclass[11pt]{amsart}   
\usepackage[english]{babel}
\usepackage{mathtools}
\usepackage{amssymb,amscd, graphics,color,latexsym,cancel}   
\usepackage[linktoc=all, pagebackref, hyperindex]{hyperref}%
\hypersetup{colorlinks,  citecolor=blue,  filecolor=blue,  linkcolor=blue,  urlcolor=black}

\usepackage{tikz-cd} 
\usepackage{extpfeil}
\usepackage{tikz}
\usetikzlibrary{matrix,calc}
\usepackage{mathrsfs}
\usepackage[all]{xy}
\usepackage{amsfonts}
\usepackage[normalem]{ulem}
\usepackage{euscript}
\usepackage{amsmath}
\usepackage{ifpdf}
\usepackage{cleveref}
\LARGE\textwidth=6in
\newcommand{\rar}{\rightarrow}
\newcommand{\lar}{\longrightarrow}

\newtheorem{mthm}{Main Theorem}
\newtheorem{Theorem}{Theorem}[section]
\newtheorem{Lemma}[Theorem]{Lemma}
\newtheorem{Corollary}[Theorem]{Corollary}
\newtheorem{Proposition}[Theorem]{Proposition}

\theoremstyle{definition}
\newtheorem{Remark}[Theorem]{Remark}
\newtheorem{Example}[Theorem]{Example}

\newtheorem{Definition}[Theorem]{Definition}
\newtheorem{Question}[Theorem]{Question}
\def\sqr#1#2{{\vcenter{\hrule height.#2pt
			\hbox{\vrule width.#2pt height#1pt \kern#1pt
				\vrule width.#2pt}
			\hrule height.#2pt}}}
\def\phi{\varphi}

\def\VaVa{{\mathcal V}\kern-5pt {\mathcal V}}
\def\gr#1#2{{\rm gr}\, _{#1}(#2)}
\def\gr{{\rm gr}\,}
\def\hht{{\rm ht}\,}
\def\depth{{\rm depth}\,}
\def\ass{{\rm Ass}\,}
\def\Min{{\rm Min}\,}
\def\codim{{\rm codim}\,}
\def\spec{{\rm Spec}\,}

\def\ker{{\rm ker}\,}
\def\grade{{\rm grade}\,}
\def\rk{\rm rank}

\def\syz{\mbox{\rm Syz}}
\def\cok{\mbox{\rm coker}}

\def\To#1#2#3#4{{\rm Tor}\,^{#1}_{#2}({#3},{#4})}
\def\Ext#1#2#3#4{{\rm Ext}\,^{#1}_{#2}({#3},{#4})}
\def\spec#1{{\rm Spec}\, (#1)}
\def\proj#1{{\rm Proj}\, (#1)}
\def\supp#1{{\rm Supp}\, (#1)}

\def\ini{\mbox{\rm in}}

\def\der#1{\mbox{\rm der}_k(#1)}

\def\cl#1{{\mathcal #1}}

\def\Fitt{\text{Fitt}}

\def\phi{\varphi}

\def\hht{{\rm ht}\,}
\def\grade{{\rm grade}\,}

\def\gg{{\bf g}}

\def\ZZ{{\bf Z}}
\def\ff{{\bf f}}

\def\ff{{\bf f}}

\def\fm{{\mathfrak m}}
\def\fn{{\mathfrak n}}

\def\fp{{\mathfrak p}}
\def\fq{{\mathfrak q}}

\def\fm{{\mathfrak m}}
\def\fn{{\mathfrak n}}

\def\cl#1{{\cal #1}}
\def\rk{\rm rank}

\newcommand{\excise}[1]{}

\def\NZQ{\mathbb}               

\def\ZZ{{\NZQ Z}}

\def\PP{{\NZQ P}}

\def\Jc{{\mathcal J}}

\def\G{{\mathcal G}}

\def\Q{{\mathcal Q}}

\def\N{{\mathcal N}}
\def\opn#1#2{\def#1{\operatorname{#2}}} 
\opn\chara{char} \opn\length{\lambda} \opn\pd{pd} \opn\rk{rk}
\opn\projdim{proj\,dim} \opn\injdim{inj\,dim} \opn\rank{rank}
\opn\depth{depth} \opn\grade{grade} \opn\height{height}
\opn\embdim{emb\,dim} \opn\codim{codim}

\opn\Tr{Tr} \opn\bigrank{big\,rank}
\opn\superheight{superheight}\opn\lcm{lcm}
\opn\trdeg{tr\,deg}
	\opn\reg{reg} \opn\lreg{lreg} \opn\ini{in} \opn\lpd{lpd}
	\opn\size{size} \opn\sdepth{sdepth}
	\opn\link{link}\opn\fdepth{fdepth}\opn\lex{lex}
	\opn\tr{tr}
	\opn\type{type}
	\opn\div{div} \opn\Div{Div} \opn\cl{cl} \opn\Cl{Cl}
	\opn\Spec{Spec} \opn\Supp{Supp} \opn\supp{supp} \opn\Sing{Sing}
	\opn\Ass{Ass} \opn\Min{Min}\opn\Mon{Mon}
	\opn\Ho{H}  \opn\indeg{indeg} \opn\Hs{HS}
	\opn\Ann{Ann} \opn\Rad{Rad} \opn\Soc{Soc}  \opn\der{Der}
	\opn\Bour{Bour} 
	\opn\Im{Im} \opn\Ker{Ker} \opn\Coker{Coker} \opn\Am{Am}
	\opn\Hom{Hom} \opn\Tor{Tor} \opn\Ext{Ext} \opn\End{End}
	\opn\Aut{Aut} \opn\id{id}
	
	\opn\nat{nat}
	\opn\pff{pf}
	\opn\Pf{Pf} \opn\GL{GL} \opn\SL{SL} \opn\mod{mod} \opn\ord{ord}
	\opn\Gin{Gin} \opn\Hilb{HP}\opn\sort{sort}
	\opn\PF{PF}\opn\Ap{Ap}
	\opn\aff{aff} \opn
	\con{conv} \opn\relint{relint} \opn\st{st}
	\opn\lk{lk} \opn\cn{cn} \opn\core{core} \opn\vol{vol}  \opn\inp{inp} \opn\nilpot{nilpot}
	\opn\link{link} \opn\star{star}\opn\lex{lex}\opn\set{set}
	\opn\width{wd}
	\opn\Fr{F}
	\opn\QF{QF}
	\opn\G{G}
	\opn\type{type}\opn\res{res}
	\opn\log{Log}
	\opn\gr{gr}   
	
	\def\pot#1#2{#1[\kern-0.28ex[#2]\kern-0.28ex]}

	\opn\dirlim{\underrightarrow{\lim}}
	\opn\inivlim{\underleftarrow{\lim}}
\begin{document}
		
		\title[The Bourbaki degree of a matrix]{The Bourbaki degree of the  syzygy module of  2 $\times$ 4   matrices} 
		\author{Marcos Jardim}
		\address{Universidade Estadual de Campinas (UNICAMP) \\ Instituto de Matem\'atica, Estat\'{\i}stica e Computação Cient\'{\i}fica (IMECC) \\ Departamento de Matem\'atica \\
			Rua S\'ergio Buarque de Holanda, 651\\ 13083-970 Campinas-SP, Brazil}
		\email{jardim@unicamp.br}
				\author{Felipe  Monteiro}
		\address{Universidade Estadual de Campinas (UNICAMP) \\ Instituto de Matem\'atica, Estat\'{\i}stica e Computação Cient\'{\i}fica (IMECC) \\ Departamento de Matem\'atica \\
			Rua S\'ergio Buarque de Holanda, 651\\ 13083-970 Campinas-SP, Brazil \newline \indent
        Universit\'e Bourgogne Europe (UBE), CNRS, Institut de Mathématiques de Bourgogne UMR 5584, F-21000 Dijon, France}
		\email{ffcmonteiro00@gmail.com}
		\author{Abbas Nasrollah Nejad}
		\address{Department of Mathematics, Institute for Advanced Studies in Basic Sciences (IASBS), Zanjan 45137-66731, Iran}
		\email{abbasnn@iasbs.ac.ir}
		
		\subjclass[2020]{Primary: 13A02, 13D02, 13H15. Secondary: 14B05, 14H20, 14H50}   	
		
		\keywords{Bourbaki ideal, graded free resolution, Buchsbaum--Rim complex, free divisor, codimension one distributions on projective spaces}

		\begin{abstract}
		We introduce and study the Bourbaki degree as a numerical invariant for \(2 \times 4\) matrices $\Theta$ of homogeneous polynomials over a polynomial ring \(R = k[x_1, \dots, x_n]\). This invariant, defined via a Bourbaki sequence for the syzygy module \(\operatorname{Syz}(\Theta)\), generalizes previous constructions for plane curves and Jacobian matrices. Our main result is an explicit formula expressing the Bourbaki degree in terms of the degrees of the rows, the initial degree of a syzygy, and the first two Hilbert coefficients of the cokernel module \(\mathcal{Q} = \operatorname{coker}(\Theta)\). We apply this framework to two important cases. First, matrices with constant first row, which are determined by a three-equigenerated ideal \(J = (f_1, f_2, f_3)\), where we show the Bourbaki degree measures how far \(J\) is from being a perfect ideal, and we completely characterize its smaller and larger values. Second, for a linear matrix, we use the Kronecker--Weierstrass classification to determine all possible Bourbaki degrees and homological types. This classification reveals the existence of a linear matrix with Bourbaki degree equal to 2, a value that does not occur for Jacobian matrices. Finally, in the geometric context of \(\mathbb{P}^3\), we provide a sufficient condition for \(\operatorname{Syz}(\Theta)\) to define a codimension one distribution and obtain bounds on the Bourbaki degree when the initial degree is small.
		\end{abstract}
		\maketitle
    		
    	\begingroup
            \hypersetup{linkcolor=black}
        \endgroup
		
		\section*{Introduction}

The logarithmic tangent sheaves and differentials with logarithmic poles on divisors have been studied since the foundational works by Deligne \cite{Deligne1970} and Saito \cite{Saito}. This has given rise to a wide range of problems and approaches at the intersection of combinatorics, commutative algebra, algebraic geometry, and complex analysis. Over projective spaces, these objects may be defined as the kernel of a gradient vector $\nabla(f)$ for a homogeneous polynomial $f \in k[x_1, \ldots, x_n]$. This construction relates the singularities of the divisor $V(f)$ and the associated module of logarithmic differentials, or equivalently, the module of syzygies for the matrix $\nabla(f)$. 

The homological point of view relates the structure of the singularities of $V(f)$ to properties of the resolution of the Jacobian ideal $J_f = \langle \partial_1 f, \ldots, \partial_n f \rangle$. This perspective appears in many articles
\cite{dimca1992singularities,dimca2020plane,faenzi2022stability,faenzi2014logarithmic,terao1981generalized}. The simplest possible resolution occurs when the syzygy module is free, and divisors with this property are called \emph{free}. This classical framework has been recently adapted to include complete intersections of codimension $>1$ \cite{Faenzi2025} and toric varieties \cite{FJM24}.

In another direction, Jardim, Nejad, and Simis \cite{MAA} introduced the notion of the \textit{Bourbaki degree} for plane curves (case $n=3$), a discrete invariant that vanishes precisely for free curves, and measures how \textit{non-free} a plane curve is. { For recent developments and applications of this invariant, see \cite{RZAS,dimca2025some, Pokora-Dimca-Abe} }This construction is based on classical \textit{Bourbaki sequences} \cite{bourbaki1965diviseurs}, which relates modules to ideals via a choice of a free submodule, and has been used in several contexts (\cite{dimca2020jumping,evans1976bourbaki,herzog2021graded}). It is therefore natural to seek applications of this construction to the context of hypersurfaces (see \cite{DimcaSticlaru24}) and complete intersection curves (done by Monteiro in \cite{monteiro2025} for in $\mathbb{P}^3$).

In this paper, we generalize and unify these approaches by studying an arbitrary $2\times4$ matrix $\Theta$ of homogeneous polynomials over the polynomial ring $R=k[x_1, \ldots, x_n]$ defining a morphism of graded $R$-modules 
$$ \Theta ~:~ R^4 \longrightarrow R(d_1) \oplus R(d_2), $$
without assuming, as in \cite{Faenzi2025,monteiro2025}, that the lines of $\Theta$ are gradients of homogeneous polynomials. Then  $\syz(\Theta) \doteq \Ker(\Theta)$ is the \textit{syzygy module} of the matrix $\Theta$; we assume that $\rk(\syz(\Theta))=2$.

The matrix $\Theta$ is said to be \textit{free} if $\syz(\Theta)=R(-e_1)\oplus R(-e_2)$ for some non-negative integers $e_1$ and $e_2$.

Otherwise, set $e=\indeg(\syz(\Theta))$, and choose a homogeneous syzygy $\nu \in \syz(\Theta)$ of degree $e$; it induces an injective morphism $R(-e)\stackrel{\nu}{\hookrightarrow}\syz(\Theta)$ whose cokernel is a torsion-free module of rank one, hence isomorphic to an ideal $I_\nu$ of $R$, up to grading; that is $\cok(\nu)\simeq I_\nu(s)$. The \textit{Bourbaki degree} of $\Theta$, denoted $\Bour(\Theta)$, is defined to be the degree of $R/I_\nu$; see further details in Section \ref{sec:Bourbaki-degree}.

In addition, let $\Q\doteq\cok(\Theta)$; its Hilbert polynomial can be written as follows:
$$ \mathrm{HP}_\Q(t) = \dfrac{e_0(\Q)}{(n-2)!}t^{n-2} + \dfrac{e_0(\Q)d-2e_1(\Q)}{2(n-3)!}t^{n-3} + \mbox{(lower order terms in } t\mbox{)}. $$
The coefficients $e_0(\Q)$ and $e_1(\Q)$ will play an important role in this paper; note that $e_0(\Q)$ is a non-negative integer, while $e_1(\Q)\in\ZZ$; in addition, $e_0(\Q)=e_1(\Q)=0$ when $\dim(\Q)\le n-3$.

\begin{mthm}
Let $R=k[x_1,\ldots,x_n]$ be a polynomial ring over an infinite field $k$ and $\Theta$ be a $2\times 4$ matrix of rank $2$ in $R$ whose first and second rows consist of homogeneous polynomials of degrees $d_1$ and $d_2$, respectively. Let $\nu \in \syz(\Theta)$ be a minimal homogeneous generator of degree $e\geq 1$. If $\Theta$ is not free, then  
\begin{equation}\label{BourDeg-intro}
\Bour(\Theta)= \left\{ \begin{array}{ll}
(e-d)(e+e_0(\Q))+\fq_{\Theta}+\ell_\Q+e_1(\Q), & \textrm{if } e_0(\Q) \neq 0; \\
e(e-d)+\fq_{\Theta}-e_1(\Q), & \textrm{if } e_0(\Q) = 0.
\end{array}\right.
\end{equation}
where $d=d_1+d_2$ , $\fq_{\Theta}=d_1^2+d_2^2+d_1d_2$,  
$ \ell_\Q= \tfrac{1}{2}\bigl(e_0(\Q)^2+e_0(\Q)\bigr)$. Furthermore, $\Bour(\Theta)=\fq_{\Theta}$ when $\dim(\Q)\le n-3$, and the minimal free graded resolution of $\Q$ is a Buchsbaum--Rim complex.
\end{mthm}

We also observe in Remark \ref{Bourbounds} that  
\[
\fq_\Theta + \ell_{\Q} + e_1(\Q) - \frac{(d + e_0(\Q))^2}{4} \le
\Bour(\Theta) \leq \fq_\Theta + \ell_{\Q} + e_1(\Q),
\]
where the lower bound is attained when $e = \frac{d - e_0(\Q)}{2}$, while the upper bound is attained when $e=d$. However, fixed $d_1$ and $d_2$, not all intermediate values in the above interval are attained by some matrix $\Theta$; in fact, when $n=3$, there is no Jacobian matrix $\Theta$ with $d_1=d_2=1$ such that $\Bour(\Theta)=2$. The existence of gaps for the value of $\Bour(\Theta)$ is reminiscent of the gaps for the Bourbaki degree for plane algebraic curves of degrees $3$ and $4$ found in \cite[Section 5]{Futata2023}.


{
We then turn to two important special cases. The first is the case where the first row of $\Theta$ has degree $0$ and the second row has degree $d>0$, which leads to the study of three--equigenerated ideals and to Main Theorem~\ref{BourJ}. In this setting, the Bourbaki degree introduced in~\cite{MAA} for the Jacobian ideals of reduced plane curves appears as a special case.
}
{ \begin{mthm}\label{BourJ}
Let $J=(f_1,f_2,f_3)\subset R=k[x_1,\ldots,x_n]$ be a three-equigenerated ideal with $\deg(f_i)=d$ and $\gcd(f_1,f_2,f_3)=1$, and let
$e=\indeg(\syz(J))$. If $J$ is an almost complete intersection, then
\[
\deg(R/J)+\Bour(J)=d^2+e^2-ed.
\]
Moreover, the following hold:
\begin{enumerate}
\item $\Bour(J)=0$ if and only if $J$ is perfect.
\item $\Bour(J)=1$ if and only if the Bourbaki ideal is  a complete intersection generated by two independent linear forms.
\item $\Bour(J)=2$ if and only if the Bourbaki ideal is either the intersection of two codimension-two linear primes, or a codimension-two linear primary ideal of multiplicity two, or a complete intersection of type $(1,2)$.
\item $\Bour(J)=d^2-1$ if and only if $e=d$ and $\deg(R/J)=1$.
\item $\Bour(J)=d^2$ if and only if $J$ is a complete intersection.
\end{enumerate}
Assume furthermore that $\syz(J)$ is locally free on the punctured spectrum. Then
\[
\Bour(J)\le e^2 \quad \mbox{and }\quad  \ d(d-e)\le \deg(R/J)\le d^2+e^2-ed.
\]
In the quadratic case, if $J$ has height two and $\deg(f_i)=2$, then
\begin{enumerate}
\item[{\rm (a)}] if $n=3$, then $\Bour(J)\neq 2$;
\item[{\rm (b)}] if $n\geq 4$, $\Theta$ is locally free, and $J$ is saturated, then $\Bour(J)\neq 2$.
\end{enumerate}
\end{mthm}

The second is the linear case $d_1=d_2=1$, treated in \Cref{sec:linear-matrices}, where Main Theorem~\ref{mainBourlinear} describes the possible Bourbaki degrees and homological types via the Kronecker--Weierstrass classification. }

{
\begin{mthm}\label{mainBourlinear}
    Let $\Theta$ be a $2\times 4$ matrix of linear forms in $R=k[x_1,\ldots,x_n]$. Then the
Kronecker--Weierstrass normal form determines completely the Bourbaki degree $\Bour(\Theta)$,
the initial degree $\indeg(\syz(\Theta))$, and the homological type of $\Q=\cok(\Theta)$.
In particular, each normal form is either free, or nearly free with $\Bour(\Theta)=1$, or of
Buchsbaum--Rim type with $\Bour(\Theta)=3$, except for one exceptional type, namely
\[
\Theta=D_2\mid B_1=
\begin{bmatrix}
x_1 & x_2 & 0 & x_3\\
0   & x_1 & x_2 & x_4
\end{bmatrix},
\]
for which $\Bour(\Theta)=2$ and $\Q$ has a minimal free resolution which is not of
Buchsbaum--Rim type. 
\end{mthm}}

It is worth noting that the exceptional case in the previous statement yields an example of a linear matrix with Bourbaki degree $2$, a value which does not occur when $\Theta$ is the Jacobian of a pencil of quadrics, according to the classification done in \cite[Section 6]{Faenzi2025}. This classification is studied from our point of view in \Cref{thm:pencils-quadrics-FJV}.

{ A natural direction for further investigation is the case $d_1=1$ and $d_2\geq 2$. This is the first genuinely mixed-degree setting, and it is natural to ask which Bourbaki degrees occur and how they encode the algebraic and geometric properties of the ideal $I_2(\Theta)$. This case is of particular geometric interest, since it includes the ideals associated with complete intersection curves in $\PP^3$ defined by a quadric and a form of degree $d+1$. Thus, its study may provide a bridge between the numerical theory of Bourbaki degrees and the geometry of $(2,d+1)$--complete intersection space curves.}

Finally, we shift our attention to a more geometric context, relating, in the case $n=3$, the syzygy module $\syz(\Theta)$ to the theory of distributions on the three-dimensional projective space. To be more precise, let
$$ \varepsilon \doteq \begin{bmatrix}
x_1 \\ x_2 \\ x_3 \\ x_4
\end{bmatrix}: R(-1) \longrightarrow R^4 $$
denote the \emph{Euler vector}, and consider the composition $\Theta \circ \varepsilon$; it can be written in terms of two homogeneous polynomials as follows
$$ \Theta \circ \varepsilon =
\begin{bmatrix}
h_1\\h_2
\end{bmatrix}: R(-1) \longrightarrow R(d_1)\oplus R(d_2). $$

\begin{mthm}
If $(h_1, h_2)$ is a regular sequence, then $\syz(\Theta)(1)$ is the graded module associated to the tangent sheaf of a codimension one distribution on $\mathbb{P}^3$.
\end{mthm} 

In particular, the hypothesis always holds when $\Theta$ is the Jacobian matrix of a regular sequence of homogeneous polynomials $(f,g)$, and we show it holds whenever the ideal of minors of $\Theta$ has codimension at least two (see \Cref{rmk:distributions-normal-matrices}). Using the geometry of foliations, we obtain more refined bounds for the Bourbaki degree when the initial degree of the syzygy module is low (Propositions \ref{prop:e-0-foliations} and \ref{prop:initial-degree-1-n-4}). We also show that nearly free matrices (i.e. $\Bour(\Theta)=1$) cannot induce locally free syzygy modules on the punctured spectrum, extending a result known for Jacobian matrices.

\bigskip 

The paper is organized as follows.
\Cref{sec:hom-struct} collects the homological and numerical ingredients that make possible the definition
of the Bourbaki degree. We begin by studying the module $\Q=\cok(\Theta)$ and its relation with the determinantal ideal $I_2(\Theta)$. When $I_2(\Theta)$ has maximal possible grade, the \textit{Buchsbaum--Rim complex}
provides an explicit graded free resolution of $\Q$, whose maps are determined by $\Theta$ and its $2$-minors. This gives the fundamental homological model for the matrices
considered here. We then introduce the notions of free and locally free matrices, and study
how these properties are reflected in the support and associated primes of $\Q$. A second
topic  is the initial degree $e=\indeg(\syz(\Theta))$, which is shown to satisfy
$0\leq e\leq d_1+d_2$ and is further constrained by the interaction between the two row-wise
syzygy modules. Finally, we study the Hilbert polynomial of $\Q$ and its first Hilbert
coefficients $e_0(\Q)$ and $e_1(\Q)$, which are the numerical invariants entering the main
formula of the paper. In particular, we prove the bound $0\leq e_0(\Q)\leq d_1+d_2$, and
show that the extremal case forces $\syz(\Theta)\simeq R^2$ up to grading (see~\Cref{Maximale0}).

\medskip 

\Cref{sec:Bourbaki-degree} contains the algebraic core of the paper. There, we define the Bourbaki degree of $\Theta$ and establish the formula
$$
\Bour(\Theta)=(e-d)(e+e_0(\Q))+\fq_\Theta+\ell_\Q+e_1(\Q),
$$
where $d=d_1+d_2$, $\fq_\Theta=d_1^2+d_2^2+d_1d_2$, and $\ell_\Q=\frac{1}{2}(e_0(\Q)^2+e_0(\Q))$.
This identity generalizes the Bourbaki degree formula previously obtained in the Jacobian
setting in dimension three~\cite[Theorem 2.1]{MAA}, and shows that $\Bour(\Theta)$ is governed simultaneously by
the degree of a minimal syzygy and by the Hilbert-theoretic behavior of the cokernel. In the case $\dim \Q\leq n-3$, the formula simplifies to $\Bour(\Theta)=\fq_\Theta$, recovering the Buchsbaum--Rim situation. We also prove that graded free resolutions of $\syz(\Theta)$ and
of the Bourbaki ideal determine one another (see Theorem~\ref{Thm:Bourbaki-resolutions}), which allows us to relate the homological
structure of $\syz(\Theta)$ to the geometry of the associated codimension two scheme. We show that local freeness of $\Theta$ forces the Bourbaki scheme to be locally Cohen--Macaulay on the punctured spectrum. As
applications, we characterize nearly free matrices and $3$-syzygy matrices in terms of the associated Bourbaki ideal. 

\medskip 

{ In \Cref{sec:Bour-of-ideals},  we consider the case where the first row of $\Theta$ has degree $0$ and the second row has degree $d>0$. This allows us to extend the notion of Bourbaki degree introduced in~\cite{MAA} for Jacobian ideals of reduced plane curves to the broader setting of three--equigenerated ideals. Indeed, in this case $\Theta$ is graded equivalent to a matrix of the form
\[
\begin{bmatrix}
0&0&0&1\\
f_1&f_2&f_3&0
\end{bmatrix},
\qquad f_1,f_2,f_3\in R_d.
\]
Motivated by this normal form, given an ideal
\[
J=(f_1,f_2,f_3)\subset R=k[x_1,\ldots,x_n]
\]
generated by three homogeneous forms of the same degree $d$, with $\gcd(f_1,f_2,f_3)=1$, we associate to $J$ the matrix
\[
\Theta=
\begin{bmatrix}
0&0&0&1\\
f_1&f_2&f_3&0
\end{bmatrix}.
\]}
Then $I_2(\Theta)=J$, and the matrix construction developed in \Cref{sec:hom-struct} induces a Bourbaki degree for $J$, denoted by $\Bour(J):=\Bour(\Theta)$. In this way, we extend to arbitrary three-equigenerated ideals the construction introduced in \cite{MAA} for the Jacobian
ideals of reduced plane curves, following a strategy analogous to \cite[Section 2.2]{monteiro2025}. If $J$ is an almost complete intersection and $e=\indeg(\syz(J))$, then
$$
\deg(R/J)+\Bour(J)=d^2+e^2-ed.
$$
Thus $\Bour(J)$ measures the defect of $\deg(R/J)$ from the value $d^2+e^2-ed$ determined by the numerical data $d$ and $e$. The main structural result of the section is Theorem~\ref{Bourequi}, which gives a precise interpretation of the extremal and small values of the Bourbaki degree. In particular, it shows that $\Bour(J)=0$ if and only if $J$ is perfect ideal, that $\Bour(J)=d^2$ if and only if $J$ is a complete intersection, and that the cases $\Bour(J)=1$ and $\Bour(J)=2$ are governed by the geometry of the associated Bourbaki ideal. In this sense, $\Bour(J)$ emerges as a numerical invariant that measures how far $J$ is from the perfect case, and especially from being a complete intersection.

Moreover, assuming that $\syz(J)$ is locally free on the punctured spectrum, Theorem~\ref{thm:equigenerated-Bour-bound} gives the bound $\Bour(J)\le e^2$, and hence
$$ d(d-e)\le \deg(R/J)\le d^2+e^2-ed. $$
A particularly significant application appears in Theorem~\ref{prop:loc-free-d-2-Bour2-eq-ideal}, which treats the quadratic case. It shows that if $J$ has height two, is generated by quadrics, and $\syz(J)$ is locally free on the punctured spectrum, then $\Bour(J)=2$ if and only if $n\geq 4$ and $J$ is unsaturated. Hence, the value $\Bour(J)=2$ is completely characterized in this setting: it can occur only in higher dimensions, and precisely when the ideal is not saturated. In particular, this rules out $n=2$; equivalently, in the Jacobian situation, there are no plane cubics with Bourbaki degree $2$. Therefore, even such a small nonzero value of the Bourbaki degree already reflects a subtle geometric defect and reveals that the phenomenon is inherently higher-dimensional.

\medskip

\Cref{sec:linear-matrices} is devoted to the case of linear matrices. Here, the Kronecker--Weierstrass normal
form provides a complete classification of $2\times 4$ matrices of linear forms up to equivalence. We analyze each possible normal form and determine the corresponding
Bourbaki degree, the initial degree of syzygies, and the homological behavior of the cokernel. This produces a rather complete picture: a large family of normal forms falls into the Buchsbaum--Rim case and has $\Bour(\Theta)=3$; several other families are free; some
are nearly free with $\Bour(\Theta)=1$; and one exceptional type, namely the matrix $D_2|B_1$, has $\Bour(\Theta)=2$ and a minimal free resolution which is not of
Buchsbaum--Rim type. For $n=4$, Jacobian matrices were classified in \cite[Theorem 6.1]{Faenzi2025}. In that context, the pairs $\sigma=(f,g)$ are interpreted as \emph{pencils of quadrics}, and the classification shows that the only possible Bourbaki degrees are ${0,1,3}$. In contrast, our classification (see \Cref{thm:KWform-Bour2}) produces an example with Bourbaki degree $2$, and hence one that cannot be Jacobian. This example is particularly noteworthy because $\syz(\Theta)$ is locally free on the punctured spectrum even though $\Bour(\Theta)\neq 0$, a phenomenon excluded in the Jacobian setting by the additional Jacobian hypothesis.

\medskip

In \Cref{sec:geometrical-point-of-view}, we turn to the geometric point of view, focusing on $n=4$ where we have the analogy with Jacobian matrices and logarithmic sheaves. Removing the Jacobian hypothesis, we study when the sheaf associated to $\syz(\Theta)$ defines a codimension one distribution on $\mathbb{P}^3$, and reproduce the strategy to consider the sub-foliation by curves induced by the syzygy of minimal degree (from \cite[Section 3]{monteiro2025}) to obtain bounds for the Bourbaki degree for sufficiently low values of the initial degree of the module $\syz(\Theta)$. At the end, we show that nearly free matrices cannot induce locally free modules at the punctured spectrum, extending a known result for Jacobian matrices (\cite[Proposition 19]{monteiro2025}).

The present work lays the groundwork for a broader application of the Bourbaki degree as a numerical invariant for matrices and ideals. We hope that its interplay with homological algebra, commutative algebra, and algebraic geometry will open up new avenues for research, including the study of logarithmic sheaves, distributions and foliations, and the classification of singularities in higher dimensions.

\subsection*{Acknowledgments}
FM is supported by the São Paulo Research Foundation (FAPESP) by the PhD grant number \#2021/10550-4, under the cotutelle supervision of Marcos Jardim and Daniele Faenzi, with partial funding by the Bridges "Brazil-France interplays in Gauge Theory, extremal structures and stability" projects ANR-21-CE40-0017 and ANR-17-EURE-0002.
MJ is partially supported by the CNPQ grant number 305601/2022-9, the Brazilian Centre for Geometry (FAPESP-CEPID Project number 2024/00923-6), and the PAPESP-ANR project number 2021/04065-6.
We thank Daniele Faenzi for insightful discussions. ANN was partially supported by FAPESP Grant No. 2022/09853-5 and sincerely thanks IMECC–UNICAMP for its generous hospitality and for providing an excellent academic environment during his visit. This work was carried out during that period and benefited greatly from the stimulating research atmosphere at IMECC.


\section{Homological Structure}\label{sec:hom-struct}

Throughout, let $R = k[x_1,\ldots,x_n]$ be a polynomial ring with $n\geq 3$ over an algebraically closed field $k$ of characteristic zero. Let 
		$$\Theta=\begin{bmatrix}
			f_1&f_2&f_3&f_4\\
			g_1&g_2&g_3&g_4
		\end{bmatrix}$$ 
		be a $2\times 4$ matrix of rank $2$ with entries in $R$, where $\deg(f_i)=d_1$ and $\deg(g_i)=d_2$ for all $i$  and $0\le d_1\le d_2$. Additionally, assume that the greatest common divisor (gcd) of the entries in the first row and the second row is $1$. 	
	 The matrix $\Theta$ determines a graded $R$-linear map of free modules
	\begin{equation}\label{Theta}
	   R^{4} \xrightarrow{\;\Theta\;} R(d_{1}) \oplus R(d_{2}), 
	\end{equation}
which sends the basis element $e_i \in R^{4}$ to the pair $(f_i,\,g_i)$ for each $i=1,\ldots,4$.

We denote by $\N:=\mathrm{Im}(\Theta)$  and $\Q:=\cok(\Theta)$ the image  and the cokernel of $\Theta$, respectively. The kernel of the graded $R$-linear map~\eqref{Theta} is the second syzygy
module of $\Q$, and we denote it by
\[
\syz(\Theta) := \ker(\Theta) \subseteq R^{4}.
\]
We then have the following short exact sequences:
\begin{equation}\label{FundSequ}
0\rar \syz(\Theta)\rar R^4\stackrel{\Theta}{\lar} \N\rar 0, \,\, \ \ \ \  \quad 0\rar \N\rar R(d_1)\oplus R(d_2)\rar \Q\rar 0. 
\end{equation}
 Passing to associated sheaves on $\mathbb{P}^{n-1} = \proj{R}$, 
 the graded module $\syz(\Theta)$ determines a coherent subsheaf
$
\widetilde{\syz(\Theta)} \subset \mathcal{O}_{\mathbb{P}^{n-1}}^{\oplus 4}$. 
By analogy with \textit{logarithmic tangent sheaf} or \textit{tangential idealizer} associated to a hypersurface, the graded $R$-module $\syz(\Theta)$ and its associated sheaf $\widetilde{\syz(\Theta)}$ may be regarded as the \textit{tangential module} and logarithmic tangent sheaf associated to $\Theta$, respectively. Moreover,  $\syz(\Theta)$ is the second syzygy of the $R$-module $\Q$ and hence $\syz(\Theta)$ is a reflexive graded $R$-module of rank $2$; equivalently, $\widetilde{\syz(\Theta)}$ is a rank two reflexive sheaf on $\mathbb{P}^{n-1}$.

\subsection{Graded free resolution}\label{subsec:graded-free-res}
Since $\mathrm{Fitt}_0(\Q) =I_2(\Theta)\subseteq \Ann(\Q)$ (\cite[Proposition 20.7]{EisenbudBook}), and moreover they have the same support $V(\Ann(\Q))=\supp(\Q)=\supp(R/I_2(\Theta))$. By the determinant formula for the height of a determinantal ideal~\cite[Theorem 2.1]{Bruns-Vetter}, one has $\hht(I_2(\Theta))\leq 3$. Then 
		\[
		1\leq \dim \Q=\dim R/\Ann(\Q)=\dim R/I_2(\Theta)\leq n-1. 
		\]
		  
		Assume that $\grade\big(I_2(\Theta)\big)\geq3$ (e.g,. $\dim \Q=1$). 
		By the \emph{Buchsbaum--Rim complex}~\cite[A2.6]{EisenbudBook}, which resolves the cokernel of an $n\times m$ matrix with $n\ge m$ when the ideal of maximal minors has the maximal possible grade, the module $\mathcal{Q}$ admits a graded free resolution of the form
        
\begin{equation}\label{BRCmain}
\begin{split}
0 \to
R\big(-2(d_1+d_2)\big)\oplus R\big(-(d_1+3d_2)\big)
\xrightarrow{\phi}
R^{4}\big(-(d_1+2d_2)\big) \\
\xrightarrow{\psi}
R^{4}(-d_2)
\xrightarrow{\Theta}
R\oplus R(d_1-d_2)
\to \mathcal{Q} \to 0.
\end{split}
\end{equation}
Here, the left-most map is (up to sign) the near transpose of $\Theta$, while the entries of the middle map are the $2$-minors of $\Theta$.  More precisely, 
		\[\psi=\begin{bmatrix}
		0 & \Delta_{12} & \Delta_{13} & \Delta_{14} \\
		-\Delta_{12} & 0 & \Delta_{23} & \Delta_{24} \\
		-\Delta_{13} & -\Delta_{23} & 0 & \Delta_{34} \\
		-\Delta_{14} & -\Delta_{24} & -\Delta_{34} & 0
	\end{bmatrix},\quad 
		\phi=
		\begin{bmatrix}
			 f_1 & -g_1\\[2pt]
			- f_2 & g_2\\[2pt]
			f_3 & -g_3\\[2pt]
			-f_4 & g_4
		\end{bmatrix},
		\]
where $\Delta_{ij}=f_ig_j-f_jg_i$ are the $2$-minors of the matrix $\Theta$. 
Under the above grade assumption, the ideal $I_2(\Theta)$ is perfect (hence Cohen--Macaulay), and its resolution has the same length as that of $\mathcal{Q}$ in \eqref{BRCmain}.

The second syzygy module of $\Q$, namely $\syz(\Theta)=\ker(\Theta)$, has rank $2$.
It is free if and only if $\Q$ has projective dimension $2$; equivalently, by the
Auslander--Buchsbaum formula, $\depth(\Q)=n-2$. In particular, if $\syz(\Theta)$ is free,
then $\hht\big(I_2(\Theta)\big)\leq 2$.
The converse fails in general. Indeed, outside the maximal-grade case, the
Buchsbaum--Rim complex may degenerate and produce other minimal free resolutions of $\Q$.
For example, one may have resolutions of the form
\[
0\longrightarrow R \longrightarrow R^{3} \longrightarrow R^{4}
\xrightarrow{\ \Theta\ } R^{2} \longrightarrow \Q \longrightarrow 0,
\]
or
\begin{equation}\label{depth_zero}
0\longrightarrow R \longrightarrow R^{4} \longrightarrow R^{5} \longrightarrow
R^{4} \xrightarrow{\ \Theta\ } R^{2} \longrightarrow \Q \longrightarrow 0,
\end{equation}
corresponding to $\depth(\Q)=n-3$ and $\depth(\Q)=n-4$, respectively.

Moreover, even when $\hht\big(I_2(\Theta)\big)=2$, the module $\Q$ may admit a minimal free resolution of the same homological length as in the Buchsbaum--Rim case, without
being the Buchsbaum--Rim resolution itself. This phenomenon does not occur for linear matrices, as will be seen in \Cref{sec:linear-matrices}. For instance, let
\[
\Theta=
\begin{pmatrix}
-x_2 & x_3-x_1 & x_2-x_4 & -x_3\\
x_2^2 & 2x_1x_2 & 3x_3^2+2x_3x_4 & x_3^2
\end{pmatrix},
\]
a Jacobian matrix from \cite[Example 62]{monteiro2025}. In this case, $\Q$ has minimal graded free resolution
\[
0 \to R(-4) \to R(-3)^3 \oplus R(-2) \to R^4
\xrightarrow{\Theta} R(1)\oplus R(2) \to \Q \to 0,
\]
so $\depth(\Q)=2$. Thus, although the resolution has the same homological length as in the Buchsbaum--Rim case, the graded shifts show that it is not the Buchsbaum--Rim resolution.
\begin{Example}
Let $\Theta$ be the matrix 
\[
\Theta = D_2 \mid B_1 = \begin{bmatrix}
x_1 & x_2 & 0   & x_3\\
0   & x_1 & x_2 & x_4
\end{bmatrix}.
\]
We will show at \Cref{thm:KWform-Bour2} that $\Q$ has a resolution of shape~\eqref{depth_zero}.
\end{Example} 
We next define the notions of freeness and local freeness for the matrix $\Theta$.
\begin{Definition}\label[Definition]{free-locallyfree}
The matrix $\Theta$ is said to be \emph{free} if $\syz(\Theta)$ is a free $R$-module. Equivalently, since $\syz(\Theta)$ is the second syzygy module of $\Q$, the Auslander--Buchsbaum formula yields $\depth(\Q)=n-2.$
We say that $\Theta$ is \emph{locally free} if $\syz(\Theta)$ is locally free on the punctured spectrum of $R$.   
\end{Definition}
The following Lemma is an algebraic version of~\cite[Lemma 2.2]{DMJ}. 
\begin{Lemma}\label[Lemma]{lem:loc-free-associated-primes}\hbox{}
\begin{enumerate}
	\item[{\rm (a)}] $\Theta$ is locally free if and only if every associated prime of $\Q$ different from $\mathfrak{m}$ has codimension at most $2$.
	\item[{\rm (b)}] If $\Theta$ is locally free, then every minimal prime of  $I_2(\Theta)$ has codimension at most $2$. 
\end{enumerate}
\end{Lemma}
\begin{proof}
For any prime $\fp\subseteq R$, localize the exact sequence
\[
0 \lar \syz(\Theta) \lar R^{4} \xrightarrow{\Theta} R^{2} \lar  \Q \lar  0
\]
at $\fp$ to obtain
\[
0 \lar \syz(\Theta)_\fp \lar R^{4}_\fp \xrightarrow{\Theta_\fp} R^{2}_\fp \lar  \Q_\fp \lar  0.
\]
Thus $\syz(\Theta)_\fp$ is free if and only if $\Q_\fp$ has projective dimension at most~2 over $R_\fp$. 
Since $R_\fp$ is regular, the Auslander--Buchsbaum formula gives
\[
\projdim_{R_\fp}(\Q_\fp) = \hht(\fp) - \depth_{R_\fp}{(\Q_p)}.
\]
Hence $\syz(\Theta)_\fp$ is free if and only if $\depth_{R_\fp}{(\Q_{\fp})}\geq \hht(\fp)-2$.

Assume $\syz(\Theta)_\fp$ is locally free on the punctured spectrum.
Let $\fq$ be an associated prime of $\Q$ with $\fq\neq \fm$. 
Then $\depth_{R_\fq}{(\Q_\fq)}=0$. By above  applied at $\fq$, one has $0=\depth_{R_\fq}{(\Q_\fq)}\geq \hht(\fq)-2$, so that $\hht(\fq)\leq 2$.

Conversely, assume every associated prime of $\Q$ has height at most~2.
Let $\fp\neq\fm$. We must show $\depth_{R_\fp}{(\Q_\fp)}\geq \hht(\fp)-2$.

If $\Q_\fp=0$, the inequality holds trivially. Assume $\Q_\fp\neq 0$.
If $\hht(\fp)\leq 2$, then $\hht(\fp)-2\leq0$ and $\depth_{R_\fp}{(\Q_\fp)}\geq 0$, so the inequality holds.

Now suppose $\hht(\fp)\geq 3$.
Since every associated prime of $\Q$ has height $\le 2$, $\fp$ itself cannot be an associated prime of $\Q$. Hence $\depth_{R_\fp}{(\Q_\fp)}\geq 1$.
Set $t=\depth_{R_\fp}{(\Q_\fp)}$. 
Choose a maximal regular sequence $\underline{x} = x_1,\dots,x_t$ in $\fp R_\fp$ on $\Q_\fp$, 
and let $\Q'=\Q_\fp / (\underline{x})\Q_\fp$. 
Then $\depth_{R_\fp}{(\Q'_\fp)}= 0$, so the maximal ideal $\fp R_\fp$ is an associated prime of $\Q'$ (since a module over a local ring has depth zero if and only if its maximal ideal is associated). 
The associated primes of $\Q'$ are among those of $\Q_\fp$, which in turn are contained in the set of associated primes of $\Q$ contained in $\fp$.
By hypothesis, these primes have height \(\le 2\). Therefore, $ \dim(\Q') \le 2$. 
On the other hand, because $\underline{x}$ is a regular sequence, $\dim(\Q')=\dim(\Q_\fp)-t\leq \hht(\fp)-t$
Combining these dimension gives $\hht(\fp)-t\leq 2$, i.e. $t\geq \hht(\fp)-2$. 
Thus $\depth_{R_\fp}{\Q_\fp}\geq \hht(\fp)-2$ which implies that $\syz(\Theta)_\fp$ is free on the punctured spectrum.

The part (b) follows from (a) and the fact that the minimal primes of $\Q$ coincide with the minimal primes of $R/I_2(\Theta)$. 
\end{proof}
\begin{Example}
Let $R=k[x_1,\ldots,x_n]$ with $n\geq 4$. 
Consider the matrix 
\[
\Theta=\begin{bmatrix}
	f_1&f_2&0&0\\
	0&0&g_1&g_2
\end{bmatrix}
\]
with $f_1,f_2$ and $g_1,g_2$ in disjoint sets of variables. One has
 $I_2(\Theta)=(f_1,f_2)\cap (g_1,g_2)$ and $\hht(I_2(\Theta))=2$.  Then $\depth(\Q)=\depth (R/(f_1,f_2)\oplus R/(g_1,g_2))=n-2$ and  the graded minimal free resolution of $\Q$ is of the form 
\[0\rar R(-(d_1+d_2))\oplus R(-2d_2)\rar R^4(-d_2)\rar R\oplus R(d_1-d_2)\rar \Q\rar 0.\]
Therefore, $\Theta$ is free. 
\end{Example}
\begin{Proposition}\label[Proposition]{dimQ-n-3-BR-resolution}
If $\dim(\Q)\leq n-3$, then $\Theta$ is not free. 
\end{Proposition}
\begin{proof}
If $\dim \Q\leq n-3$, then $\grade(I_2(\Theta))\geq 3$. Thus the module $\Q$ has a Buchsbaum--Rim  resolution~\eqref{BRCmain} and hence $\depth(\Q)=n-3$. 
\end{proof}

\subsection{The initial degree}
Recall that, for a graded module $E=\bigoplus_{i\geq 0} E_i$ over an 
$\mathbb{N}$-graded Noetherian ring, its \emph{initial degree} is defined by
\[
\indeg(E) := \min\{\, i \mid E_i \neq 0\,\}.
\]
We denote by
\[
e := \indeg(\syz(\Theta))
\]
the initial degree of the graded $R$-module $\syz(\Theta)$.
For $1\le i<j\le 4$, let $\Delta_{ij}=f_ig_j-f_jg_i$ denote the $2$-minors of $\Theta$ 
determined by columns $i$ and $j$. The columns of the skew-symmetric matrix
\begin{equation}\label{BRsyzygy}
	\begin{bmatrix}
		0 & \Delta_{12} & \Delta_{13} & \Delta_{14} \\
		-\Delta_{12} & 0 & \Delta_{23} & \Delta_{24} \\
		-\Delta_{13} & -\Delta_{23} & 0 & \Delta_{34} \\
		-\Delta_{14} & -\Delta_{24} & -\Delta_{34} & 0
	\end{bmatrix}
\end{equation}
give homogeneous elements of $\syz(\Theta)$. Since each $\Delta_{ij}$ has degree
$d_1+d_2$, this produces elements of $\syz(\Theta)$ in degree $d_1+d_2$, and
therefore
\begin{equation}\label{UPindeg}
	0 \le e \le d_1+d_2.
\end{equation}

The map $\Theta$ is the direct sum of the maps $\Theta_\mathbf{f}: R^4\rar R(d_1)$ and $\Theta_\mathbf{g}: R^4\rar R(d_2)$ defined by the individual rows $\mathbf{f}=\begin{bmatrix}
f_1&f_2&f_3&f_4
\end{bmatrix}\ \mbox{and} \  \mathbf{g}=\begin{bmatrix}
g_1&g_2&g_3&g_4
\end{bmatrix} $ of the matrix $\Theta$. We obtain the following exact sequences of graded $R$-modules
\begin{equation}\label{Mf}
0\rar \syz(\Theta_\ff)\rar R^4\stackrel{\Theta_\ff}{\lar} R(d_1)\rar R/\Jc_\ff\rar 0,
\end{equation}
 and 
 \begin{equation}\label{Mg}
 0\rar \syz(\Theta_\gg)\rar R^4\stackrel{\Theta_\gg}{\lar} R(d_2)\rar R/\Jc_\gg\rar  0,
 \end{equation}
 where $\Jc_\ff$ and $\Jc_\gg$ are the ideals generated by the first and second rows of $\Theta$, respectively. Note that $\syz(\Theta_\ff)$ and $\syz(\Theta_\gg)$ are reflexive modules of rank $3$. 
One has 
\[ \syz(\Theta)=\{\nu \in R^4\ |\  \Theta_\mathbf{f}(\nu)=0\ \mbox{and}\ \Theta_\mathbf{g}(\nu)=0 \}=\syz(\Theta_\mathbf{f})\cap \syz(\Theta_\mathbf{g}).\]
Let  $e_{\mathbf{f}}$ and $e_{\mathbf{g}}$ stand for the initial degrees of graded modules $\syz(\Theta_{\mathbf{f}})$  and $\syz(\Theta_{\mathbf{f}})$, respectively. From the equality $\syz(\Theta)=\syz(\Theta_\mathbf{f})\cap \syz(\Theta_\mathbf{g})$, we conclude that 
\[t:=\max\{{e_{\mathbf{f}}},\, e_{\mathbf{g}}\} \leq e\leq d_1+d_2.\]

\begin{Lemma}\label[Lemma]{e.max}
 If $\To{R}{1}{R/\Jc_\ff}{R/\Jc_\gg}_{t+d_1+d_2}\neq 0$, then $e=t$. 
\end{Lemma}
\begin{proof}
 we already have $t\le e$. Thus, it is enough to prove $e\leq t$. From the exact sequence
\[
0\to \Jc_f\to R\to R/\Jc_f\to 0,
\]
tensoring with $R/\Jc_g$ gives the standard graded isomorphism
\[
\Tor_1^R(R/\Jc_f,R/\Jc_g)\cong \frac{\Jc_f\cap \Jc_g}{\Jc_f\Jc_g}.
\]
Hence there exists
$
h\in (\Jc_f\cap \Jc_g)_{t+d_1+d_2}\setminus (\Jc_f\Jc_g)_{t+d_1+d_2}.
$
Write
\[
h=\sum_{i=1}^4 f_i a_i=\sum_{i=1}^4 g_i b_i,
\qquad
a_i\in R_{t+d_2},\ b_i\in R_{t+d_1}.
\]
Reducing modulo $\Jc_g$, we get
$\sum_{i=1}^4 \overline{f_i}\,\overline{a_i}=0
$ in $R/\Jc_g$.

\medskip
\noindent
\textit{Claim.} If $\syz(\Theta)_t=0$, then $a_i\in \Jc_g$ for all $i$.

\smallskip
\noindent
\textit{Proof of the claim.}
If not, then $\overline a\neq 0$ in $(R/\Jc_g)^4$. Writing
\[
a_i=\sum_{j=1}^4 g_j c_{ij}+r_i,
\qquad c_{ij}\in R_t,
\]
with some $r_i\notin \Jc_g$, and substituting into
$ \sum_{i=1}^4 f_i a_i=\sum_{i=1}^4 g_i b_i$
we obtain $\sum_{i=1}^4 f_i r_i\in \Jc_g$. 
Since $\deg r_i=t+d_2$ and $\deg g_j=d_2$, comparison of the degree-$t$ coefficients yields a nonzero vector $\nu\in R_t^4$ satisfying $\sum_{i=1}^4 f_i\nu_i=0$ and $\sum_{i=1}^4 g_i\nu_i=0$ that is, $\nu\in \syz(\Theta)_t$, a contradiction. Thus $a_i\in \Jc_g$ for all $i$.

If $\syz(\Theta)_t=0$, the claim gives $h\in \Jc_f\Jc_g$, contradicting the choice of $h$. Hence $\syz(\Theta)_t\neq 0$, so $e\le t$. 
\end{proof}
\begin{Example}\label[Example]{D2B1}
 Let $R=k[x_1,x_2,x_3,x_4]$ and
\[
\Theta=
\begin{bmatrix}
x_1 & x_2 & 0 & x_3\\
0 & x_1 & x_2 & x_4
\end{bmatrix}.
\]
Since the third entry of the first row and the first entry of the second row are zero, it follows that the canonical vectors $e_3$ and $e_1$ are syzygies of $\Theta_\ff$ and $\Theta_\gg$, respectively. Then $e_\ff=e_\gg=0$. One check that $(\Jc_\ff\cap \Jc_\gg)_2\subseteq \Jc_\ff\Jc_\gg$, hence the condition of Lemma~\ref{e.max} fails, and consequently $e\neq 0$. Now we show that $e=2$. 

Recall that $\syz(\Theta)=\syz(\Theta_\ff)\cap \syz(\Theta_\gg)$
where $\ff$ and $\gg$ denote the first and second rows of $\Theta$, respectively. We first describe these syzygy modules.

For the first row, the syzygy module $\syz(\Theta_\ff)$ consists of all $(a,b,c,d)\in R^4$ such that
$x_1a+x_2b+x_3d=0.$ Since the third entry of $\ff$ is zero, the element $e_3=(0,0,1,0)$
belongs to $\syz(\Theta_\ff)$. Moreover, the ideal $(x_1,x_2,x_3)$ is a complete
intersection, so its first syzygy module is generated by the Koszul syzygies.
Lifting these to $R^4$ gives the following homogeneous generators:
\[
s_1=e_3,\,\ \   s_2=(x_2,-x_1,0,0),\,\ \   s_3=(x_3,0,0,-x_1),\,\ \  s_4=(0,x_3,0,-x_2).
\]
Thus $\syz(\Theta_\ff)$ is generated in degrees $0$ and $1$, with initial degree $e_\ff=0$. 

For the second row the syzygy module $\syz(\Theta_\gg)$ consists of all $(a,b,c,d)\in R^4$ such that
$x_1b+x_2c+x_4d=0$. Here the first entry of $g$ is zero, so $e_1=(1,0,0,0)\in \syz(\Theta_\gg)$.
Again, $(x_1,x_2,x_4)$ is a complete intersection, yielding generators
\[
t_1=e_1,\, \ \ t_2=(0,x_2,-x_1,0),\, \ \ t_3=(0,x_4,0,-x_1),\, \ \ t_4=(0,0,x_4,-x_2).
\]
Hence $\syz(\Theta_\gg)$ is also generated in degrees $0$ and $1$, with initial degree $e_\gg=0$

A constant vector $(a,b,c,d)\in k^4$ lies in $\syz(\Theta_\ff)$ if and only if  $a=b=d=0$,
and it lies in $\syz(\Theta_\gg)$ if and only if  $b=c=d=0$. Together this forces
$a=b=c=d=0$, hence $\syz(\Theta)_0=0$. 

The space $\syz(\Theta_\ff)_1$ is spanned by $s_2,s_3,s_4$. We compute
\[
\Theta_\gg(s_2)=-x_1^2,\quad
\Theta_\gg(s_3)=-x_1x_4,\quad 
\Theta_\gg(s_4)=x_1x_3-x_2x_4.
\]
A linear combination $\alpha s_2+\beta s_3+\gamma s_4$ with $\alpha,\beta,\gamma\in k$
lies in $\syz(\Theta_\gg)$ if and only if 
\[
\Theta_\gg(\alpha s_2+\beta s_3+\gamma s_4)
=
-\alpha x_1^2-\beta x_1x_4+\gamma x_1x_3-\gamma x_2x_4
=0.
\]
Since the monomials $x_1^2,x_1x_4,x_1x_3,x_2x_4$ are linearly independent in $R_2$,
we must have $\alpha=\beta=\gamma=0$. Therefore
$\syz(\Theta)_1=0$ which implies that $e=2$.
\end{Example}

We finish this section with the following result, reminiscent of the notion of \emph{compressibility} for sequences (\cite[Section 2.4]{Faenzi2025}), showing $\Theta$ is free when the initial degree is zero.

\begin{Proposition}\label[Proposition]{prop:compressibility}
Let $\Theta$ be a $2 \times 4$ matrix of rank two with entries in $R$. Then $e = \indeg(\syz(\Theta)) = 0$ if and only if there is a change of coordinates such that $\Theta$ can be expressed as a matrix with a zero column:
$$
\Theta = \begin{pmatrix}
0 & f_2' & f_3' & f_4'\\
0 & g_2' & g_3' & g_4'
\end{pmatrix}.
$$
In particular, one obtains $\syz(\Theta) \simeq R \oplus R(s)$ for some $s \geq 0$.
\end{Proposition}

\begin{proof}
If $\syz(\Theta) \simeq R \oplus R(-s)$ for some $s \geq 0$, then it follows that $e = \indeg(\syz(\Theta)) = 0$. Therefore, there is a constant non-trivial syzygy, given by $\nu = (a,b,c,d)^T$ such that $\Theta \cdot \nu = 0$. Now, consider a change of coordinates of $k^4$ which takes $\nu$ to the vector $e_1 = (1,0,0,0)$. The matrix $\Theta$ in this new coordinates satisfies $\Theta \cdot e_1 = 0$, and we obtain the form of the matrix of the claim. So one has a block matrix
$$
\Theta = \begin{pmatrix}
0 & \Theta'
\end{pmatrix}
$$
where $\Theta'$ is a $2 \times 3$ matrix with homogeneous entries. Forming the following commutative diagram with exact rows, where $M_\nu \doteq \cok(\nu)$,
\begin{center}
\begin{tikzcd}
            & R \arrow[d, "\nu"] \arrow[r, Rightarrow, no head] & R \arrow[d, "e_1"]                           &                                                             &   \\
0 \arrow[r] & \syz(\Theta) \arrow[r] \arrow[d, two heads]       & R^4 \arrow[r, "\Theta"] \arrow[d, two heads] & \mathcal{N}_\Theta \arrow[r] \arrow[d, Rightarrow, no head] & 0 \\
0 \arrow[r] & M_\nu \arrow[r]                                   & R^3 \arrow[r, "\Theta'"]                     & \mathcal{N}_{\Theta} \arrow[r]                              & 0
\end{tikzcd}
\end{center}
one concludes $\rank(M_\nu) = 1$ and $M_\nu$ is reflexive, as $\mathcal{N}_{\Theta}$ is torsion-free. Hence $M_\nu \simeq R(s)$ for some $s \in \mathbb{Z}$, the sequence splits so $\syz(\Theta) \simeq R \oplus R(s)$ and the condition about the initial degree gives $s \geq 0$.
\end{proof}

\subsection{Hilbert polynomials and multiplicities}

A quick deflection to Hilbert multiplicities (see, for example, \cite[Section 4.1]{bruns1998cohen}). Let $M=\oplus_{m\geq 0}M_m$ be a graded $R=k[x_1,\ldots,x_n]$-module of dimension $d>0$. The Hilbert function of $M$ is defined by $\mathrm{HF}_M(m)=\dim_kM_m$. For sufficiently large $m$, this function agrees with a polynomial of degree $d-1$ given by
\begin{align*}
\mathrm{HP}_M(t) &= e_0(M) \binom{t + d-1}{d-1} - e_1(M) \binom{t+d-2}{d-2} + \ldots + (-1)^{d-1} e_{d-1}(M)\\
&= \dfrac{e_0(M)}{(d-1)!}t^{d-1} + \dfrac{e_0(M)d-2e_1(M)}{2(d-2)!}t^{d-2} + \mbox{(lower order terms in $t$)}
\end{align*}
with $e_0 \in \mathbb{Z}$ positive integer called the \emph{multiplicity} or \emph{degree} of $M$, commonly denoted by $\deg(M)$ and $e_1,\ldots,e_{d-1}\in \mathbb{Z}$ are the \emph{Hilbert coefficients} of $M$.

The Hilbert series of a finitely generated graded $R$-module $M$ of dimension $d$ can be written as
\[
H_M(t)
= \sum_{i \in \mathbb{Z}} \dim_{k}(M_i)\, t^i
= \frac{h(t)}{(1 - t)^d},
\]
where $h(t) \in \mathbb{Z}[t,t^{-1}]$ and $h(1) \neq 0$.
The Hilbert coefficients satisfy (see \cite[Proposition 4.1.9]{bruns1998cohen})
\[
e_0(M) = h(1)
\qquad \text{and} \qquad
e_1(M) = h'(1).
\]
These coincide with the usual Hilbert coefficients obtained from the Hilbert polynomial of $M$. We will be mostly interested in the case $M = \Q$ of the cokernel of a $(2\times 4)$-matrix of rank two. In general, the dimension of $\Q$ is at most $d = n-1$, as we impose the condition of maximal rank. Therefore, we fix $d = n-1$, and obtain three possible behaviours:
\begin{itemize}
    \item[(a)] $\codim(\Q) = 1$ if and only if $e_0(\Q) \neq 0$;
    \item[(b)] $\codim(\Q) = 2$ if and only if $e_0(\Q) = 0$ and $e_1(\Q) \neq 0$.
    \item[(c)] $\codim(\Q) \geq 3$ if and only if $e_0(\Q) = e_1(\Q) = 0$, where the Buchsbaum--Rim resolution is minimal, as we explored before.
\end{itemize}

A useful formula to compute Hilbert coefficients concretely is the \emph{associativity formula} (see, for example, \cite[Theorem 14.7]{Matsumura1987}). For $d = \dim(\Q)$, let
$$
\Min_d(\Q) \doteq \{\fp \in \Ann(\Q) : \dim(R/\fp) = d\},
$$
so the formula, when $\dim(\Q) = n-1$, can be written as:
$$
e_0(\Q) = \sum_{\fp \in \Min_{n-1}(\Q)} \length_{R_\fp}(\Q_\fp) \cdot \deg(R/\fp).
$$
Analogously, when $\Q$ has dimension $n-2$, then $e_0(\Q) = 0$ and
$$
e_1(\Q) = \sum_{\fp \in \Min_{n-2}(\Q)} \length_{R_\fp}(\Q_\fp) \cdot \deg(R/\fp).
$$

The first two Hilbert coefficients $e_0(\Q)$ and $e_1(\Q)$ will be the essential discrete invariants for our formula for the Bourbaki degree, obtained in the next section. The following basic result ensures that these stay constant after changing to a polynomial ring with more variables, and we add a proof for the sake of completeness. 

\begin{Lemma}\label[Lemma]{Lem:ext-of-var}
Assume that $\Theta$ is a matrix with entries in the polynomial ring $R = k[x_1, \ldots, x_m]$, and denote $M = \Ker(\Theta)$ as an $R$-module. Consider the same matrix with entries in a ring $S = k[x_1, \ldots, x_n]$ with $m \leq n$, so that $\widetilde{M} = \Ker(\Theta \otimes_k S)$. If $d = \dim(M)$, then $e_0(M) = e_0(\widetilde{M})$ and $e_1(M) = e_1(\widetilde{M})$.

Moreover, if $
K_{\bullet} \to M$ is a minimal free resolution for $M$ over $R$, then $K_\bullet \otimes S \to \widetilde{M}$ is a minimal free resolution for $\widetilde{M}$. 
\end{Lemma}

\begin{proof}
First, since $S$ is flat over $R$, it follows that $\widetilde{M} \simeq M \otimes_{k} S$. Moreover, from this identity, the Hilbert series is the product:
$$
\Hs_{\widetilde{M}}(t) = \Hs_{M}(t) \cdot \Hs_{S}(t)=\dfrac{h(t)}{(1-t)^d}\cdot\dfrac{1}{(1-t)^{n-m}}=\dfrac{h(t)}{(1-t)^{d+n-m}}.
$$
If $\fp=\Ann_R{M}$, then $\Ann_S{M}=\fp S$ and $\dim S/\fp S=\dim R/\fp+n-m=d+(n-m)$. Thus $\dim_S \widetilde{M}=d+(n-m)$ and the dominator is exactly $\dim(\widetilde{M})$. 
Therefore, by the same definition, 
\[
e_0(\widetilde{M})=h(1)=e_0(M), \quad e_1(\widetilde{M})=h'(1)=e_1(M). 
\]

For the second statement, let
\[
K_\bullet : 0 \lar F_r \lar F_{r-1} \lar \cdots \lar F_1 \lar F_0 \lar M \lar 0
\]
be a minimal graded free resolution of $M$ over $R$.
Minimality means that for each $i$ the differential
$d_i: F_i \lar F_{i-1}$ is represented by a matrix whose entries lie in the homogeneous maximal ideal $\fm=(x_1,\dots,x_m)$ of $R$.

Since $S$ is flat over $R$, tensoring $K_\bullet$ with $S$ yields an exact complex
\[
K_\bullet \otimes_R S : 0 \lar F_r \otimes_R S \lar \cdots
\lar F_0 \otimes_R S \lar M \otimes_R S
\lar 0.
\]
For each $i$, we have $F_i \otimes_R S \cong S^{\beta_i},$
with the same rank $\beta_i$ as $F_i$.
The differentials are given by the same matrices, now viewed over $S$; their entries lie in $\fm \subset \fn$
where $\fn_S=(x_1,\dots,x_n)$
is the homogeneous maximal ideal of $S$.
Hence, the complex remains minimal over $S$. Finally,
$ M \otimes_R S \cong \widetilde{M}$,
so $K_\bullet \otimes_R S \to \widetilde{M}$ is a minimal graded free resolution of $\widetilde{M}$ over $S$.
\end{proof}

Finally, we obtain a bound for $e_0(\Q)$ in terms of the degrees $d_1$ and $d_2$, and characterize when this bound is attained. We give an alternative proof for this fact at \Cref{subsec:semistability-bounds}, using $\mu$-semistability of torsion-free sheaves in $\mathbb{P}^{n-1}$.

\begin{Theorem}\label[Theorem]{Maximale0}
Let $\Theta$ be a $2\times 4$ matrix of rank $2$ in $R$ whose first and second rows consist of homogeneous polynomials of degrees $d_1$ and $d_2$, respectively. Then $0 \leq e_0(\Q)\leq d \doteq d_1 + d_2$. Moreover, if  $e_0(\Q) = d$, then there is a graded $R$-module isomorphism of degree zero $\syz(\Theta) \simeq R^2$.
\end{Theorem}
\begin{proof}
	If $\dim \Q\le n-2$, then $e_0(\Q)=0$ and the first assertion is clear. Hence assume $\dim \Q=n-1$. Recall that $\mathrm{Fitt}_0(\Q)=I_2(\Theta)$, and each $2\times2$ minor of $\Theta$ has degree $d=d_1+d_2$.
		
		Let $\{ \fp_i \}$ be the height--one minimal primes of $I_2(\Theta)$. Since 		$R=k[x_1,\ldots,x_n]$ is a UFD, each $\fp_i$ is principal, say $\fp_i=(h_i)$, where
		$h_i$ is an irreducible homogeneous polynomial. Set
		\[
		m_i:=\length_{R_{\fp_i}}\!\big(\Q_{\fp_i}\big).
		\]
		These lengths are positive since $\fp_i\subset \supp(\Q)$.		
		Localize at $\fp_i$. Then $R_{\fp_i}$ is a DVR with uniformizer $h_i$. The module $\Q_{\fp_i}$ is finitely generated and supported only at the maximal ideal $\fp_iR_{\fp_i}$, hence $\dim \Q_{\fp_i}=0$ and therefore $\Q_{\fp_i}$ is a torsion
		$R_{\fp_i}$--module and hence . For a finitely generated torsion module $M$ over a DVR with uniformizer $h$, one has
		\[
		\mathrm{Fitt}_0(M)=\big(h^{\length(M)}\big).
		\]
		Applying this to $M=\Q_{\fp_i}$ yields
		\[
	\mathrm{Fitt}_0(\Q_{\fp_i})=\big(h_i^{m_i}\big).
		\]
		Since Fitting ideals commute with localization, we obtain
		\[
		I_2(\Theta)R_{\fp_i}
		=\mathrm{Fitt}_0(\Q)R_{\fp_i}
		=\mathrm{Fitt}_0(\Q_{\fp_i})
		=\big(h_i^{m_i}\big).
		\]
		Therefore every $2\times2$ minor $\Delta$ of $\Theta$ lies in $\big(h_i^{m_i}\big)$ in $R_{\fp_i}$, hence $h_i^{m_i}$ divides $\Delta$ in $R_{\fp_i}$. Consequently, for each $i$, 		$h_i^{m_i}$ divides $\Delta$, so $H:=\prod_i h_i^{m_i}$ divides $ \Delta$. Thus 
		\begin{equation}\label{degH}
       	\deg(H)=\sum_i m_i\,\deg(h_i)\ \le\ \deg(\Delta)=d.
		\end{equation}
				
		Since $\dim \Q=n-1$, the associativity formula yields
		\[
		e_0(\Q)=\sum_{\substack{\hht(\fp)=1, \ \fp\in \ass(\Q)}}
		\length_{R_\fp}(\Q_\fp)\,e_0(R/\fp).
		\]
		Here the relevant primes are exactly the height--one minimal primes $\fp_i=(h_i)$ of $I_2(\Theta)$, and for such a principal prime one has $e_0(R/(h_i))=\deg(h_i)$. Hence
		\begin{equation}\label{e0}
			e_0(\Q)=\sum_i m_i\,\deg(h_i).	
		\end{equation}
	
		Combining~\eqref{degH} and~\eqref{e0} gives $e_0(\Q)\le d$.
	\medskip
		Assume now that $e_0(Q)=d$. Then formulas~\eqref{degH} and~\eqref{e0}  force $\deg(H)=\deg(\Delta)=d$, and therefore each minor
		$\Delta$ satisfies $\Delta=c_\Delta H$ for some scalar $c_\Delta\in k$. Now fix any prime $\fq \subset R$. We show that $\syz(\Theta)_\fq$ is free of rank $2$.
		
		If $\fq \not\supseteq (H)$, then $H$ is a unit in $R_\fq$, hence some minor $\Delta$ is a unit in $R_\fq$. Thus 
		$
		\Theta_\fq : R_\fq^{4} \longrightarrow R_\fq^{2}
		$
		is surjective. Over a local ring, any surjection between free modules splits, so the exact sequence
		\[
		0 \lar \syz(\Theta)_\fq
		\lar R_\fq^{4}
		\lar R_\fq^{2}
		\lar 0
		\]
		splits and therefore $\syz(\Theta)_\fq \simeq R_\fq^{2}$.
		
		If $\fq \supseteq (H)$, then $\fq$ contains some height--one prime
		$\fp_i = (h_i)$. Localize at $\fp_i$. The ring $R_{\fp_i}$ is a DVR. Over a PID (hence over a DVR), every submodule of a free module is free. Since $\syz(\Theta)_{\fp_i} \subset R_{\fp_i}^{4}$,
		we conclude that $\syz(\Theta)_{\fp_i}$ is free. Its rank is
		$4-2=2$, hence $\syz(\Theta)_{\fp_i} \simeq R_{\fp_i}^{2}$.
		Localizing further from $R_{\fp_i}$ to $R_\fq$ preserves freeness, so $\syz(\Theta)_\fq \simeq R_\fq^{2}$. 		
		Thus $\syz(\Theta)$ is locally free of rank $2$, i.e.\ a finitely generated projective $R$-module of rank $2$. By the Quillen--Suslin theorem~\cite[Theorem 4.59]{Rotman}, every finitely generated projective module over a polynomial ring over a field is free. Therefore
	     $\syz(\Theta) \simeq R^{2}$ (up to grading shifts in the graded category). This completes the proof.
\end{proof}



\section{The Bourbaki degree}\label{sec:Bourbaki-degree}
In this section, we associate to a $2\times4$ matrix $\Theta$ of homogeneous polynomials a numerical invariant, called the \emph{Bourbaki degree}, arising
from the Bourbaki ideal of the first syzygy module of $\Theta$.

Let $\nu \in \syz(\Theta)$ be a homogeneous syzygy of degree $e$.  
It induces an injective map of graded modules
\[
R(-e)\hookrightarrow \syz(\Theta),
\]
yielding the commutative diagram
\begin{equation}\label{diag1}
	\begin{split}
		\xymatrix@R-2ex{
			& 0 \ar[d] & 0 \ar[d] & \\
			& R(-e)\ar[d]^{\nu}\ar@{=}[r] & R(-e)\ar[d]^-{\widetilde{\nu}} \\
			0 \ar[r] & \syz(\Theta)\ar[r] \ar[d]^u & R^{4} \ar[r]^-{\Theta} \ar[d] & N \ar@{=}[d] \ar[r] & 0 \\
			0 \ar[r] & M_{\nu} \ar[r] \ar[d] & R^{4}/R(-e) \ar[r] \ar[d] & N \ar[r] & 0 \\
			& 0 & 0 &
		}
	\end{split}
\end{equation}
The module $M_{\nu}$ is torsion-free of rank one, hence isomorphic to an ideal of $R$; this ideal is called the \emph{Bourbaki ideal} of $\syz(\Theta)$. Note that if $M_\nu$ is a proper ideal, then  $M_{\nu}$ has codimension two and there exists an ideal $I_{\nu}\subseteq R$ such that $M_\nu\simeq I_\nu(s)$ for some integer $s$. The projective scheme $Y:=\proj{R/I_\nu}$ has dimension $n-3$ and the Hilbert polynomial of $R/I_\nu$ is of the form 
\begin{equation}\label{HPBour}
	\mathrm{HP}_{R/I_\nu}(t)=\dfrac{\deg(R/I_\nu)}{(n-3)!}t^{n-3}+ \dfrac{\deg(R/I_\nu)(n-2) - 2e_1(R/I_\nu)}{2(n-4)!}t^{n-4}+\ldots+e_{n-3}(R/I_\nu).
\end{equation}
Also, since the dimension of the cokernel of $\Theta$ is at most $n-1$, it follows that 
\begin{equation}\label{HPQ}
\mathrm{HP}_{\Q}(t)=\dfrac{e_0(\Q)}{(n-2)!}t^{n-2} +\dfrac{e_0(\Q)(n-1) - 2e_1(\Q)}{2(n-3)!}t^{n-3}+\cdots+e_{n-2}(\Q),
\end{equation}
with $e_0(\Q)\geq 0$ and $e_0(\Q)>0$ if and only if $\dim \Q=n-1$. Note that if $\dim \Q=r<n-1$,  then $e_0(\Q)=e_1(\Q)=\cdots=e_{n-r-2}(\Q)=0$ and $e_{n-r-1}(\Q)>0$. 
  
\begin{Theorem}\label[Theorem]{Bourideal}
Let $R=k[x_1,\ldots,x_n]$ be a polynomial ring over an infinite field $k$ and $\Theta$ be a $2\times 4$ matrix of rank $2$ in $R$ whose first and second rows consist of homogeneous polynomials of degrees $d_1$ and $d_2$, respectively. Let $\nu \in \syz(\Theta)$ be a minimal homogeneous generator of degree $e\geq 1$. Then 
\begin{itemize}
	\item[{\rm (a)}]
$M_{\nu}$ is free if and only if $\Theta$ is free. In this case, $\dim(\Q)=n-2$, $\Q$ is maximal Cohen-Macaulay and  $\Q$ admits the graded minimal free resolution
\[0\rar R(-(e+d_2))\oplus R(-(d_1+2d_2-e))\rar R^4(-d_2)\rar R\oplus R(-(d_2-d_1))\rar \Q \rar 0. \] 
\item[{\rm (b)}] If $\Theta$ is not free, then $M_\nu$ is isomorphic to a proper homogeneous ideal $I_\nu\subseteq R$  of codimension two, such that the induced isomorphism $M_\nu\simeq I_\nu(s)$ is homogeneous of degree zero, where $s=e-d+e_0(\Q)$. Furthermore, 
\begin{equation}\label{BourDeg}
	\deg(R/I_\nu)=(e-d)(e+e_0(\Q))+\fq_{\Theta}+\ell_\Q+e_1(\Q),
\end{equation}
where $d=d_1+d_2$ , $\fq_{\Theta}=d_1^2+d_2^2+d_1d_2$,  
$ \ell_\Q= \tfrac{1}{2}\bigl(e_0(\Q)^2+e_0(\Q)\bigr)$. If $e_0(\Q) = 0$, then
$$
\deg(R/I_\nu)=e(e-d)+\fq_{\Theta}-e_1(\Q).
$$
In this case, we remark that our convention changes the sign of $e_1(\Q)$ to be a positive leading coefficient for the Hilbert polynomial of $\Q$.
\end{itemize}
\end{Theorem}
\begin{proof}
(a)  The first statement follows from the definition and the left vertical exact sequence in~\eqref{diag1}. Assume that $M_\nu$ is free. Then $\syz(\Theta)\simeq R(-e)\oplus R(s)$ for some integer $s$. Then the resolution $\Q$ is of the form 
\[0\rar  R^2\stackrel{\phi}{\rar} R^4 \stackrel{\Theta}{\rar} R^2\rar \Q\rar 0. \]
Thus $\hht(I_2(\Theta))= 2$. 
By Auslander--Buschbaum formula, the projective dimension of $\Q$ is $\projdim_R(\Q)=2$. Thus $\dim (\Q)=n-2$. The same argument as in part(b) will show that $s=e-(d_1+d_2)$. Then $\Q$ has a graded minimal free resolution of the form 
\[0\rar R(-(e+d_2))\oplus R(-(d_2-s))\rar R^4(-d_2)\rar R\oplus R(-(d_2-d_1))\rar \Q \rar 0. \] 

(b) Applying additivity of Hilbert polynomials to the last exact sequence in diagram~\eqref{diag1}, and using~\eqref{HPQ} and~\eqref{HPBour}, we obtain
{\small
\[
\begin{aligned}
	0
	&=
	\mathrm{HP}_{R^{4}}(t)
	-\mathrm{HP}_{R(-e)}(t)
	-\mathrm{HP}_{I_{\nu}(s)}(t)
	-\mathrm{HP}_{N}(t)
	\\[6pt]
	&=
	4\binom{t+n-1}{n-1}
	-\binom{t-e+n-1}{n-1}
	-\binom{t+s+n-1}{n-1}
	-\binom{t+d_1+n-1}{n-1}
	-\binom{t+d_2+n-1}{n-1}
	\\
	&\quad
	+\mathrm{HP}_{R/I_{\nu}}(t)
	+\mathrm{HP}_\Q(t)
	\\[6pt]
	&=
	4\binom{t+n-1}{n-1}
	-\binom{t-e+n-1}{n-1}
	-\binom{t+s+n-1}{n-1}
	-\binom{t+d_1+n-1}{n-1}
	-\binom{t+d_2+n-1}{n-1}
	\\
	&\quad
	+\dfrac{e_0(\Q)}{(n-2)!}t^{n-2}+\left(\dfrac{\deg(R/I_v)}{(n-3)!}+\dfrac {e_0(\Q)(n-1) - 2e_1(\Q)}{2(n-3)!}\right)t^{n-3}+\ \mbox{(lower order terms in $t$)}
		\\[6pt]
	&=\left(\frac{\,e - s - d_1 - d_2 + e_0(\Q)\,}{(n-2)!}\right)t^{n-2}+Bt^{n-3}+\ \mbox{(lower order terms in $t$)},
\end{aligned}
\]}
where 
\[
B=
\frac{1}{2 (n-3)!}
\Bigl[
n(e - s - d_1-d_2)
- \bigl(e^2 + s^2 + d_1^2 + d_2^2\bigr)
+ 2\deg(R/I_v)
+ (n-1)e_0(\Q)
- 2e_1(\Q)
\Bigr].
\]

Set $d=d_1+d_2$. The vanishing of the coefficient of $t^{n-2}$ is equivalent to $e-s-d+e_0(\Q)=0$, hence $s=e-d+e_0(\Q)$. Substituting this value of $s$, we compute
$e-s-d=-e_0(\Q)$
and
\[
s^2=(e-d+e_0(\Q))^2
=e^2+d^2+e_0(\Q)^2-2ed+2e\,e_0(\Q)-2d\,e_0(\Q).
\]
Therefore,
\[
e^2+s^2+d_1^2+d_2^2
=
2e^2+d^2+d_1^2+d_2^2
-2ed
+2e\,e_0(\Q)
-2d\,e_0(\Q)
+e_0(\Q)^2.
\]

Since the Hilbert polynomial vanishes identically, the coefficient \(B\) must be zero, and we obtain
\[
\begin{aligned}
0
&=
n(e-s-d)
-\bigl(e^2+s^2+d_1^2+d_2^2\bigr)
+2\deg(R/I_v)
+(n-1)e_0(\Q)
-2e_1(\Q) \\
&=
-e_0(\Q)
-\bigl(e^2+s^2+d_1^2+d_2^2\bigr)
+2\deg(R/I_v)
-2e_1(\Q).
\end{aligned}
\]
Thus,
\[
2\deg(R/I_v)
=
e_0(\Q)
+\bigl(e^2+s^2+d_1^2+d_2^2\bigr)
+2e_1(\Q).
\]
Substituting the previous expansion and using \(d=d_1+d_2\), we obtain
\[
\deg(R/I_v)
=
e^2-ed
+\frac{d^2+d_1^2+d_2^2}{2}
+(e-d)e_0(\Q)
+\frac{e_0(\Q)^2+e_0(\Q)}{2}
+e_1(\Q).
\]
Finally, since $\frac12(d^2+d_1^2+d_2^2)=d_1^2+d_2^2+d_1d_2$, 
we conclude that
\[
\deg(R/I_v)
=
(e-d)(e+e_0(\Q))
+d_1^2+d_2^2+d_1d_2
+\frac{e_0(\Q)^2+e_0(\Q)}{2}
+e_1(\Q).
\]
\end{proof}

\begin{Definition}
Let $R=k[x_1,\ldots,x_n]$ be a polynomial ring over an infinite field $k$ and $\Theta$ be a $2\times 4$ matrix of rank $2$ in $R$ whose first and second rows consist of homogeneous polynomials of degrees $d_1$ and $d_2$, respectively. The\emph{ Bourbaki degree} of $\Theta$ is the degree of the Bourbaki ring $R/I_\nu$ of $\syz(\Theta)$ for a syzygy $\nu$ of degree $e=\mathrm{indeg}(\syz(\Theta))$. 
\end{Definition}
The Bourbaki degree does not depend on the choice of a minimal generator with initial degree. Because of this independence, we denote it by $\Bour(\Theta)$. Theorem~\ref{Bourideal} gives
\begin{equation}\label{MainBourDegFormual}
	\Bour(\Theta)=(e-d)(e+e_0(\Q))+\fq_{\Theta}+\ell_\Q+e_1(\Q).
\end{equation} 

The formula above provides a generalization of the formula for the Bourbaki degree considered previously for $2 \times 4$ Jacobian matrices for $n = 4$ in \cite[Proposition 3]{monteiro2025}, which is obtained from ours by assuming $e_0(\Q) = 0$, a hypothesis called \emph{normality} in this previous work. In \Cref{sec:Bour-of-ideals}, we explore how this notion also generalizes the one introduced for projective plane curves in \cite{MAA}.

The following result determines the Bourbaki degree provided that $\dim\Q\leq n-3$. 
\begin{Proposition}\label[Proposition]{BourM-q}
Keeping the notation and assumptions of Theorem~\ref{Bourideal}, assume that $\dim \Q\leq n-3$. Then 
$
\Bour(\Theta)=\fq_{\Theta}.
$
\end{Proposition}
\begin{proof}
Since $\dim(\Q)=\dim R/I_2(\Theta)$, the ideal $I_2(\Theta)$ has expected height $3$. Hence, we get that $\mathrm{grade}(I_2(\Theta))=3$, and $\mathcal{Q}$ admits a Buchsbaum--Rim graded free resolution of the form~\eqref{BRCmain}. It follows that the initial degree  $e=\indeg(\syz(\Theta))=d_1+d_2$ and $e_0(\Q)=e_1(\Q)=0$. Therefore, the assertion follows from the Bourbaki degree formula~\eqref{MainBourDegFormual}.
\end{proof}
The following theorem describes how the graded free resolutions of $\syz(\Theta)$ and of the Bourbaki ideal $I_\nu$ determine each other. There are analogous results for projective plane curves (see \cite{MAA}) and for $n=4$ and Jacobian matrices (see \cite{monteiro2025}).

\begin{Theorem}\label[Theorem]{Thm:Bourbaki-resolutions}
Keep notations and assumptions of Theorem~\ref{Bourideal}, and assume moreover that the matrix $\Theta$ is not free. Then, the following holds:
\begin{itemize}
\item[{\rm (a)}] Any graded free resolution of $I_\nu$,
\[
F_\bullet \longrightarrow F_0 \xrightarrow{\omega} I_\nu \longrightarrow 0,
\]
lifts to a graded free resolution of $\syz(\Theta)$ of the form
\[
F_\bullet(s)
\longrightarrow
F_0(s) \oplus R(-e)
\xrightarrow{(\omega(s),\,\nu)}
\syz(\Theta)
\longrightarrow 0.
\]
\item[{\rm (b)}]
Choose a complete set of minimal homogeneous
generators of $\syz(\Theta)$ containing $\nu$. Let
\[
F_\bullet
\longrightarrow
F_0 \oplus R(-e)
\xrightarrow{(\lambda,\nu)}
\syz(\Theta)
\longrightarrow 0
\]
be the resulting minimal graded free resolution. Then a graded minimal free
resolution of $I_\nu$ is given by
\[
F_\bullet(-s)
\longrightarrow
F_0(-s)
\xrightarrow{\lambda(-s)}
I_\nu
\longrightarrow 0.
\]
\end{itemize}
\end{Theorem}
\begin{proof}
(a) Applying the functor
$\Hom_R(F_0(s),-)$ to the short exact sequence
\[
0 \longrightarrow R(-e)
\xrightarrow{\ \nu\ }
\syz(\Theta)
\xrightarrow{\ \pi\ }
\mathcal{I}_\nu(s)
\longrightarrow 0,
\]
we obtain the exact segment
\[
\Hom_R(F_0(s),\syz(\Theta))
\xrightarrow{\ \pi^*\ }
\Hom_R(F_0(s),\mathcal{I}_\nu(s))
\longrightarrow
\Ext_R^1(F_0(s),R(-e))
=0,
\]
as $F_0(s)$ is a free $R$-module. Thus, the map $\pi^*$ is surjective. Therefore, there exists a morphism $ \widetilde{\omega}\colon F_0(s)\longrightarrow \syz(\Theta)$
such that
$
\pi\circ\widetilde{\omega}=\omega(s).
$
We now consider the map $\widetilde{\omega}\oplus\nu$ and the following
commutative diagram, whose two central columns are short exact sequences:
\[
\begin{tikzcd}
	& R(-e) & R(-e) \\
	\ker(\widetilde{\omega}\oplus\nu)
	& F_0(s)\oplus R(-e)
	& \syz(\Theta)
	& \cok(\widetilde{\omega}\oplus\nu) \\
	\ker(\omega(s))
	& F_0(s)
	& I_\nu(s)
	& 0
	\arrow[Rightarrow, no head, from=1-2, to=1-3]
	\arrow[hook, from=1-2, to=2-2]
	\arrow["\nu", hook, from=1-3, to=2-3]
	\arrow[hook, from=2-1, to=2-2]
	\arrow["\widetilde{\omega}\oplus\nu", from=2-2, to=2-3]
	\arrow[two heads, from=2-2, to=3-2]
	\arrow[two heads, from=2-3, to=2-4]
	\arrow["\pi", two heads, from=2-3, to=3-3]
	\arrow[hook, from=3-1, to=3-2]
	\arrow["\omega(s)", from=3-2, to=3-3]
	\arrow[from=3-3, to=3-4]
\end{tikzcd}
\]

By the Snake Lemma, we obtain $\cok(\widetilde{\omega}\oplus\nu)=0$ and
\[
\ker(\widetilde{\omega}\oplus\nu)\;\simeq\;\ker(\omega(s)).
\]
Consequently, extending the free resolution of $I_\nu$ and twisting by $R(s)$,
we obtain a graded free resolution of $\syz(\Theta)$:
\[
\begin{tikzcd}
	F_\bullet(s)
	& F_0(s)\oplus R(-e)
	& \syz(\Theta)
	& 0 \\
	& \ker(\omega(s))
	\arrow[from=1-1, to=1-2]
	\arrow[two heads, from=1-1, to=2-2]
	\arrow["\widetilde{\omega}\oplus\nu", from=1-2, to=1-3]
	\arrow[from=1-3, to=1-4]
	\arrow[hook, from=2-2, to=1-2]
\end{tikzcd}
\]
as claimed.

\medskip

(b) We split the graded free resolution as indicated:
\[
\begin{tikzcd}
	{F_\bullet} & {F_0\oplus R(-e)} & {\syz(\Theta)} & 0 \\
	& S
	\arrow[from=1-1, to=1-2]
	\arrow[two heads, from=1-1, to=2-2]
	\arrow["{(\lambda,\nu)}", from=1-2, to=1-3]
	\arrow[from=1-3, to=1-4]
	\arrow[hook, from=2-2, to=1-2]
\end{tikzcd}
\]
and consider the following diagram with exact rows, which induces a short exact sequence on cokernels in the last row:
\[
\begin{tikzcd}
	& R(-e) & R(-e) \\
	S & F_0\oplus R(-e) & \syz(\Theta) \\
	S & F_0 & \mathcal{I}_\nu(s)
	\arrow[Rightarrow, no head, from=1-2, to=1-3]
	\arrow["\nu", hook, from=1-2, to=2-2]
	\arrow["\nu", hook, from=1-3, to=2-3]
	\arrow[hook, from=2-1, to=2-2]
	\arrow[Rightarrow, no head, from=2-1, to=3-1]
	\arrow[two heads, from=2-2, to=2-3]
	\arrow[two heads, from=2-2, to=3-2]
	\arrow["\pi", two heads, from=2-3, to=3-3]
	\arrow[hook, from=3-1, to=3-2]
	\arrow[two heads, from=3-2, to=3-3]
\end{tikzcd}
\]
Completing this to a graded free resolution and twisting by $R(-s)$, we obtain the resolution in the claim:
\[
\begin{tikzcd}
	{F_\bullet(-s)} & {F_0(-s)} & {I_\nu} & 0 \\
	& {S(-s)}
	\arrow[from=1-1, to=1-2]
	\arrow[two heads, from=1-1, to=2-2]
	\arrow["{\lambda(-s)}", from=1-2, to=1-3]
	\arrow[from=1-3, to=1-4]
	\arrow[hook, from=2-2, to=1-2]
\end{tikzcd}
\]
\end{proof}

The next result shows that local freeness of $\Theta$ (see Definition~\ref{free-locallyfree}) forces the associated Bourbaki scheme to be locally Cohen--Macaulay away from the irrelevant maximal ideal.
\begin{Proposition}
Keeping the notation and assumptions of Theorem~\ref{Bourideal}, if $\Theta$ is locally free, then $R/I_\nu$ is locally Cohen-Macaulay on the punctured spectrum. 
\end{Proposition}
\begin{proof}
Fix $\fp\neq \fm$ and localize the Bourbaki sequence at $\fp$. Since $\syz(\Theta)_{\fp}$ is free by
hypothesis, we obtain an exact sequence
\[
0\lar R_{\fp}\lar \syz(\Theta)_{\fp}\lar (I_{\nu})_{\fp}\lar 0,
\]
hence $\projdim_{R_{\fp}}((I_{\nu})_{\fp})\le 1$ which implies that $\projdim_{R_{\fp}}((R/I_{\nu})_{\fp})\le 2$. Since $R_{\fp}$ is a regular local ring,
Auslander--Buchsbaum yields
\[
\depth (R/I_{\nu})_{\fp}
=\depth R_{\fp}-\projdim_{R_{\fp}}((R/I_{\nu})_{\fp})
\ge \dim R_{\fp}-2.
\]
On the other hand, $\dim (R/I_{\nu})_{\fp}=\dim R_{\fp}-\hht(I_{\nu}R_{\fp})\le \dim R_{\fp}-2$.
Therefore
\[
\depth (R/I_{\nu})_{\fp}\ge \dim (R/I_{\nu})_{\fp}.
\]
We obtain $\depth (R/I_{\nu})_{\fp}=\dim (R/I_{\nu})_{\fp}$, hence
$(R/I_{\nu})_{\fp}$ is Cohen--Macaulay.
\end{proof}

\begin{Example}
Let $\Theta$ be a $2\times 4$ generic matrix of linear form in $R=k[x_1,\ldots, x_n]$. The genericity assumption forces $n=8$. Thus
\[
\Theta\;=\;
\begin{bmatrix}
	x_1 & x_2 & x_3 & x_4 \\
	x_5 & x_6 & x_7 & x_8
\end{bmatrix}.
\]
The expected codimension of $I_2(\Theta)$ is $3$.
Consequently, $\dim \Q=5$ and $\grade(I_2(\Theta))=3$. Thus  $\Q$  has the Buchsbaum-Rim graded minimal free resolution
\[ 0\rar R^2(-4)\rar R^4(-3)\rar R^4(-1)\rar R^2\rar  \Q\rar  0.\]
Using this resolution, we compute the Hilbert series 
\[ \Hs_\Q(t)=\dfrac{2t+2}{(1-t)^5}.\]
It follows that, in our notation, $e_0(\Q)=e_1(\Q)=0$ and $\deg(\Q)=4$. Thus
\[
\Bour(\Theta)=\fq_{\Theta}=3.
\] 
\end{Example}

\begin{Remark}\label[Remark]{Bourbounds}
From formula~\eqref{MainBourDegFormual} and the fact that 
$0 \le e \le d = d_1 + d_2$, we obtain the following bounds for the Bourbaki degree:
\[
\Bour(\Theta) \leq \fq_\Theta + \ell_{\Q} + e_1(\Q),
\]
with equality when $e=d$. Moreover, By Theorem~\ref{Maximale0} $e_0(\Q) \leq d$, then 
\[
\Bour(\Theta) \geq \fq_\Theta + \ell_{\Q} + e_1(\Q) - \frac{(d + e_0(\Q))^2}{4},
\]
and this lower bound is attained when $e = \frac{d - e_0(\Q)}{2}$. 
\end{Remark}

\subsection{Special families of matrices}

Inspired by the literature on logarithmic tangent modules and divisors, we introduce the following definitions:

\begin{Definition}
Let $R=k[x_1,\ldots,x_n]$ be a polynomial ring over an infinite field $k$ and $\Theta$ be a $2\times 4$ matrix of rank $2$ in $R$ whose first and second rows consist of homogeneous polynomials of degrees $d_1$ and $d_2$, respectively. If $\Theta$ is not free, we say it is:
\begin{itemize}
    \item[(a)] \emph{nearly free} if $\Bour(\Theta) = 1$;
    \item[(b)] \emph{$3$-syzygy} if there is a minimal free resolution for $\syz(\Theta)$ of the form
    $$
    \ldots \to F \to \syz(\Theta) \to 0
    $$
    where $F$ has rank three. 
\end{itemize}
\end{Definition}

\begin{Proposition}\label[Proposition]{prop:nearly-free-3-syzygy}
Let $\Theta$ be a matrix as above, and assume $\Theta$ is not free. Then:
\begin{itemize}
    \item[(a)] $\Theta$ is nearly free if and only if the syzygy module admits a minimal free resolution of the form:
    $$
0 \to R(s-2) \rightarrow R(s-1)^{2} \oplus R(-e) \rightarrow \syz(\Theta) \rightarrow 0,
    $$
    where $s = e-d+e_0(\Q)$.
    \item[(b)] $\Theta$ is $3$-syzygy if and only if the associated Bourbaki ideal $I_\nu$ is a complete intersection for some choice of minimal generating syzygy $\nu$.
\end{itemize}
In particular, $(a) \Rightarrow (b)$.
\end{Proposition}

\begin{proof}
For a choice of minimal generating syzygy $\nu$ of degree $e \geq 0$, the associated Bourbaki short exact sequence
$$
0 \to R(-e) \to \syz(\Theta) \to I_\nu(s) \to 0
$$
relates free resolutions for $I_\nu$ and for $\syz(\Theta)$, as we have shown in \Cref{Thm:Bourbaki-resolutions}. For $(a)$, if $\Bour(\Theta) = 1$, then $\deg(R/I_\nu) = 1$ and it is forced that $I_\nu$ is a linear ideal of height $2$ in $R$, with a minimal free resolution given by a regular sequence of linear forms:
$$
0 \to R(-2) \to R(-1)^2 \to I_\nu \to 0,
$$
so the claimed resolution follows. Assuming the resolution for $\syz(\Theta)$ is as above, we use \Cref{Thm:Bourbaki-resolutions} to obtain a free resolution for $I_\nu$ as a complete intersection of two linear forms, and then conclude $\Bour(\Theta) = \deg(R/I_\nu) = 1$, showing it is nearly free.

For $(b)$, if $\Theta$ is $3$-syzygy, then we have a minimal free resolution of the form
$$
\ldots \to F_1 \xrightarrow{d} F \to \syz(\Theta) \to 0,
$$
so let $S \doteq \text{im } d$ denote the image of $d$ inside $F$. From the sequence, it follows that $S$ is a reflexive $R$-module of rank one, so $S$ is isomorphic to a shift of $R$. 

We may cut the free resolution at $S$, and simply write
$$
0 \to S \to F \to \syz(\Theta) \to 0
$$
for the minimal free resolution. Then, using \Cref{Thm:Bourbaki-resolutions} we relate to a free resolution for $I_\nu$ of the form
$$
0 \to R(l) \to R(-a)\oplus R(-b) \to I_\nu \to 0
$$
which describes $I_\nu$ as a complete intersection ideal inside $R$. 

Conversely, starting with a complete intersection with a minimal resolution as above, it can be lifted to a free resolution for $\syz(\Theta)$:
$$
0 \to R(l+s) \to R(s-a) \oplus R(s-b)\oplus R(-e) \to \syz(\Theta) \to 0
$$
meaning that $\Theta$ is $3$-syzygy. The minimality follows since, otherwise, $\Theta$ would be free.
\end{proof}

\subsection{Row-wise syzygy modules}\label{subsec:row-wise-syzygy}

From the inclusions $\syz(\Theta)\subseteq \syz(\Theta_\mathbf{f})$ and $\syz(\Theta)\subseteq \syz(\Theta_\mathbf{g})$ we obtain short exact sequences
\[
0\rar \syz(\Theta)\rar\syz(\Theta_\mathbf{f})\rar \Q_\mathbf{f}\rar 0,
\qquad
0\rar \syz(\Theta)\rar \syz(\Theta_\mathbf{g})\rar \Q_\mathbf{g}\rar 0,
\]
where $\Q_\mathbf{f}$ and $\Q_\mathbf{g}$ denote the respective cokernels.
\begin{Proposition}\label[Proposition]{lem:line-by-line-comparison}
Keeping the notation and assumptions as above, $\Q_\mathbf{f}$ and $\Q_\mathbf{g}$ are torsion-free graded $R$-modules of rank one. In particular, there exists graded ideals $I_\mathbf{f}$ and $I_\mathbf{g}$ of codimension at least $2$ such that 
\[\Q_\mathbf{f}\simeq I_\mathbf{f}(d_1-e_0(\Q)), \quad \Q_\mathbf{g}\simeq I_\mathbf{g}(d_2-e_0(\Q)).  \] 

Moreover, we have the formulas
\begin{align*}
\deg(R/I_{\textbf{f}}) &= \frac{e_0(\Q)^2-(2 d_1+1) e_0(\Q)}{2} - e_1(\Q) - \deg(R/\Theta_{\textbf{f}}) \\
\deg(R/I_{\textbf{g}}) &= \frac{e_0(\Q)^2-(2 d_2+1) e_0(\Q)}{2} - e_1(\Q) - \deg(R/\Theta_{\textbf{g}})
\end{align*}
\end{Proposition}
\begin{proof}
The rank assertion follows from the additivity of rank in short exact sequences. The torsion-freeness follows from the exact sequences
\[
\begin{aligned}
	&0 \longrightarrow \syz(\Theta)
	\longrightarrow \syz(\Theta_{\mathbf f})
	\xrightarrow{\; \Theta_{\mathbf g}\!\mid_{\syz(\Theta_{\mathbf f})}\;}
	R(d_2), \\[6pt]
	&0 \longrightarrow \syz(\Theta)
	\longrightarrow \syz(\Theta_{\mathbf g})
	\xrightarrow{\; \Theta_{\mathbf f}\!\mid_{\syz(\Theta_{\mathbf g})}\;}
	R(d_1).
\end{aligned}
\]
Since any torsion-free $R$-module of rank one is an ideal in $R$, it follows that there exists ideals $I_\ff(s_1)$ and $I_\gg(s_2)$ for some integers $s_1,s_2$. 
Since $\syz(\Theta)$ and $\syz(\Theta_\ff)$ are reflexive module, for all $\fp\in \spec{R}$ with $\hht(\fp)=1$ the $R_\fp$-modules $\syz(\Theta)_\fp$ and $\syz(\Theta_\ff)_\fp$ are reflexive and free as $R_\fp$ is one dimensional regular local ring. Then $(I_\ff)_\fp\simeq \syz(\Theta)_\fp/\syz(\Theta_\ff)_\fp\simeq R_\fp$ which implies that $\hht(I_\ff)\geq 2$ and $\hht(I_\gg)\geq 2$. 

To obtain the formulas for the degrees of $R/I_{\textbf{f}}$ and $\deg(R/I_{\textbf{g}})$, we use the short exact sequences involving $\syz(\Theta_\textbf{f})$, $\syz(\Theta_\textbf{g})$ and $\syz(\Theta)$. We only prove the assertions for $\Q_\ff)$. The corresponding statements for $\Q_\gg$  follow verbatim by symmetry.

From the sequence 
$$
0 \to \syz(\Theta_{\textbf{f}}) \to R^4 \xrightarrow{\Theta_{\textbf{f}}} R(d_1) \to (R/\Theta_{\textbf{f}})(d_1) \to 0
$$
we obtain one equation
\begin{align*}
\Hilb_{\syz(\Theta_{\textbf{f}})}(t) &= \Hilb_{R^4}(t) - \Hilb_{R(d_1)}(t) + \Hilb_{(R/\Theta_{\textbf{f}})(d_1)}(t)\\
&= 4 \binom{t+n-1}{n-1} - \binom{t+d_1+n-1}{n-1} + \Hilb_{(R/\Theta_{\textbf{f}})(d_1)}(t).
\end{align*}
On the other hand, we have the module $\syz(\Theta)$ and exact sequences
$$
0 \to \syz(\Theta) \to R^4 \xrightarrow{\Theta} R(d_1)\oplus R(d_2) \to \Q \to 0
$$
and 
$$
0 \to R(-e) \xrightarrow{\nu} \syz(\Theta) \to I_\nu(s) \to 0.
$$
which amount to relationships
\begin{align*}
\Hilb_{\syz(\Theta)}(t) &= \Hilb_{R^4}(t) - \Hilb_{R(d_1)}(t) - \Hilb_{R(d_2)}(t) + \Hilb_{\Q}(t)\\
\Hilb_{\syz(\Theta)}(t) &=
\Hilb_{R(-e)}(t) + \Hilb_{R(s)}(t) - \Hilb_{R/I(s)}(t)\\
&= \binom{t-e+n-1}{n-1} + \binom{t + s + n-1}{n-1} - \frac{\deg(R/I)}{(n-3)!}t^{n-3} + (\text{lower terms in t}).
\end{align*}
There is a short exact sequence, for some $s_1 \in \mathbb{Z}$, of the form
$$
0 \to \syz(\Theta) \to \syz(\Theta_{\textbf{f}}) \xrightarrow{\Theta_{\textbf{g}}} I_{\textbf{f}}(s_1) \to 0,
$$
where $I_{\textbf{f}} \subset R$ is an ideal of codimension at least two. This means that we can write its Hilbert polynomials as 
$$
\Hilb_{R/(I_{\textbf{f}})}(t) = \frac{\deg(R/I_{\textbf{f}})}{(n-3)!}t^{n-3} + (\text{lower terms in t})\\
$$
and another equation between Hilbert polynomials:
\begin{align*}
\Hilb_{\syz(\Theta_{\textbf{f}})}(t) 
&= \Hilb_{\syz(\Theta)}(t) + \Hilb_{R(s_1)} - \Hilb_{R/I_f}(t)\\
&= \binom{t-e+n-1}{n-1} + \binom{t + s + n - 1}{n-1} + \binom{t + s_1 + n - 1}{n-1}\\ 
&- \frac{\deg(R/I)}{(n-3)!} t^{n-3} - \frac{\deg(R/I_f)}{(n-3)!}t^{n-3} + (\text{lower terms in t}).
\end{align*}
The codimension of $R/\Theta_{\textbf{f}}$ is at least two, from our hypothesis, hence its dimension is at most $n-2$, and we may write
$$
\Hilb_{R/\Theta_{\textbf{f}}}(t) = \deg(R/\Theta_{\textbf{f}}) \binom{t+n-3}{n-3} + (\text{lower terms in t}),
$$
so that
\begin{align*}
\Hilb_{\syz(\Theta_{\textbf{f}})}(t) = 4 \binom{t+n-1}{n-1} - \binom{t+d_1+n-1}{n-1} +\frac{\deg(R/\Theta_{\textbf{f}})}{(n-3)!}t^{n-3} + (\text{lower terms in t}).
\end{align*}
Now, we may join the formulas together to obtain, from the identity $\Hilb_{\syz(\Theta_{\textbf{f}})}(t) = \Hilb_{\syz(\Theta)}(t) + \Hilb_{I_{\textbf{f}}(s_1)}(t)$, an equation
$$
0 = \frac{(e-d_1-s-s_1)}{(n-2)!}t^{n-2}+C t^{n-3} + (\text{lower terms in t})
$$
where
\begin{align*}
C &\doteq \frac{- (d_1^2 + e^2 + s^2 + s_1^2) + 2\deg(R/I) + 2\deg(R/I_{\textbf{f}}) + 2\deg(R/\Theta_{\textbf{f}})}{2(n-3)!}.
\end{align*}
From the term of degree $n-2$ and since $s = e-d+e_0(\Q)$, we obtain $
s_1 = d_2-e_0(\Q)$. Now, comparing the terms in
$2\deg(R/I)$ and in $d_1^2 + e^2 + s^2 + s_1^2$,
we obtain that
\begin{align*}
2\deg(R/I) - (d_1^2 + e^2 + s^2 + s_1^2) &= 2d_2 e_0(\Q) - e_0(\Q)^2 + e_0(\Q) + 2e_1(\Q)\\
&= e_0(\Q)(2d_2 - e_0(\Q) + 1) + 2e_1(\Q).
\end{align*}
Hence, from $C = 0$, we obtain the equation
\begin{align*}
-2\deg(R/I_{\textbf{f}}) &= 2\deg(R/I) - (d_1^2 + e^2 + s^2 + s_1^2) + 2\deg(R/\Theta_{\textbf{f}})\\
&= e_0(\Q)(2d_2 - e_0(\Q) + 1) + 2e_1(\Q) + 2\deg(R/\Theta_{\textbf{f}}),
\end{align*}
which gives
$$
\deg(R/I_{\textbf{f}}) = \frac{e_0(\Q)^2-(2 d_2+1) e_0(\Q)}{2} - e_1(\Q) - \deg(R/\Theta_{\textbf{f}}).
$$
\end{proof}

\begin{Remark}
Let us suppose that $\syz(\Theta)$ is free. Then $\deg(R/I_\nu) = 0$, and we obtain the identity
$$
(d-e)(e+e_0(\Q)) = \fq_{\Theta} + \ell_{\Q} + e_1(\Q).
$$
From the previous formulas, we also get a simplification
\begin{align*}
\deg(R/I_{\textbf{f}}) &= \frac{(e^2 + s^2 + s_1^2 + d_1^2)}{2} - \deg(R/\Theta_{\textbf{f}})\\
&= \frac{e_0(\Q)(e_0(\Q) - 2d_2 - 1)}{2} - e_1(\Q)
 - \deg(R/\Theta_{\textbf{f}}) \geq 0,
\end{align*}
which gives 
$e_1(\Q) \leq \frac{e_0(\Q)(e_0(\Q) - 2d_2 - 1)}{2} - \deg(R/\Theta_{\textbf{f}}) \leq \frac{e_0(\Q)(e_0(\Q) - 2d_2 - 1)}{2}$.
\end{Remark}

\begin{Remark}
The formula above, whenever $e_0(\Q) = 0$, turns into
$$
\deg(R/I_{\textbf{f}}) = e_1(\Q) - \deg(R/\Theta_{\textbf{f}}) \geq 0,
$$
matching the intuition that the degree of $R/I_{\textbf{f}}$ comes from the codimension two part of $\Q$, removing the part from the zeros of the corresponding row. Moreover, we obtain the inequality
$$
e_1(\Q) \leq \min\{\deg(R/\Theta_1), \deg(R/\Theta_2)\}.
$$

Now, assume $e_f \geq e_g$, $e_f < e$ and let $\nu_f: R(-e_f) \to \syz(\Theta_f)$ be a syzygy of minimal degree. Using the diagram
\[\begin{tikzcd}
	&& {R(-e_f)} \\
	0 & {\syz(\Theta)} & {\syz(\Theta_f)} & {I_{\textbf{f}}(d_2)} & 0
	\arrow["{\nu_f}"', from=1-3, to=2-3]
	\arrow["\phi", dashed, from=1-3, to=2-4]
	\arrow[from=2-1, to=2-2]
	\arrow[from=2-2, to=2-3]
	\arrow[from=2-3, to=2-4]
	\arrow[from=2-4, to=2-5]
\end{tikzcd}\]
we obtain an induced map $\phi$ which is injective. This corresponds to a hypersurface of degree $d_2 + e_f$ containing the closed subscheme defined by the ideal $I_{\textbf{f}}$, which in this context has degree 
$$
\deg(R/I_{\textbf{f}}) = e_1(\Q) - \deg(R/\Theta_{\textbf{f}}).
$$
If we moreover assume that $\textbf{f}$ is a regular sequence, then $\deg(R/\Theta_{\textbf{f}}) = 0$ and $e_f = d_1$, so that
$$
\deg(R/I_{\textbf{f}}) = e_1(\Q)
$$
and the closed subscheme corresponding to $R/I_{\textbf{f}}$ is contained in a hypersurface of degree $d = d_1 + d_2$.
\end{Remark}


\section{The Bourbaki degree of  three-equigenerated ideals}\label{sec:Bour-of-ideals}
{ In this section, we extend the notion of Bourbaki degree introduced in~\cite{MAA} for Jacobian ideals of reduced plane curves to the more general setting of three--equigenerated ideals. Recall that if $\Theta$ is a homogeneous $2\times 4$ matrix whose first row has degree $0$ and whose second row has degree $d>0$, then $\Theta$ is graded equivalent to a matrix of the form
\[
\begin{bmatrix}
0&0&0&1\\
f_1&f_2&f_3&0
\end{bmatrix},
\qquad f_1,f_2,f_3\in R_d.
\]
Indeed, since the first row has degree $0$, it is a nonzero vector in $k^4$, and hence, after a graded change of basis in the source $R^4$, we may assume that it is $(0,0,0,1)$. Thus $\Theta$ can be written as
\[
\begin{bmatrix}
0&0&0&1\\
f_1&f_2&f_3&f_4
\end{bmatrix}
\]
for suitable $f_1,f_2,f_3,f_4\in R_d$. Now, since $d_2>0$, multiplication by $f_4$ defines a degree-zero graded map $R\to R(d_2)$. Therefore, applying the graded automorphism
\[
\begin{bmatrix}
1&0\\
-f_4&1
\end{bmatrix}
\]
of $R\oplus R(d_2)$, we obtain
\[
\begin{bmatrix}
1&0\\
-f_4&1
\end{bmatrix}
\begin{bmatrix}
0&0&0&1\\
f_1&f_2&f_3&f_4
\end{bmatrix}
=
\begin{bmatrix}
0&0&0&1\\
f_1&f_2&f_3&0
\end{bmatrix}. 
\]
Observe, however, that this graded equivalence need not preserve the condition that the entries of the second row be relatively prime. The relevant invariant is instead the ideal of maximal minors. Thus, strictly speaking, the above construction gives a correspondence between matrices of the form
\[
\Theta=\begin{bmatrix}
0&0&0&1\\
f_1&f_2&f_3&0
\end{bmatrix}. 
\]
and ordered triples of homogeneous generators of the same degree. The associated ideal is recovered as $I_2(\Theta)=(f_1,f_2,f_3)$, so up to changing the chosen generators, this amounts to a correspondence with three--equigenerated ideals generated by forms of the same degree and having no nonconstant common divisor. In this way, the Bourbaki construction developed in the previous sections applies directly to an ideal 
$$J=(f_1,f_2,f_3)\subseteq R=k[x_1,\ldots,x_n]$$ generated by three homogeneous forms of the same degree $d$, with $\gcd(f_1,f_2,f_3)=1$. }

\begin{Proposition}\label[Proposition]{Bourequigen}
Let $J=(f_1,f_2,f_3)\subset R$ be an ideal generated by three homogeneous forms of degree $d$, with $\gcd(f_1,f_2,f_3)=1$ and let $\Theta$ be the associated matrix. Then 
\begin{equation}\label{bourthetaJ}
\Bour(\Theta)=
\begin{cases}
d^2, & \text{if $J$ is a complete intersection},\\[4pt]
e^2-ed+d^2-\deg(R/J), & \text{if $J$ is an almost complete intersection}.
\end{cases}
\end{equation}
In particular, if $J$ is the Jacobian ideal of a reduced curve $X=V(f)\subseteq \PP^2_k$ defined by the homogeneous polynomial $f\in R=k[x_1,x_2,x_3]$ of degree $d+1$, then
 \[
 \Bour(X)=e^2-ed+d^2-\tau(X),
 \] 	
 where $\tau(X)$ is the total Tjurina number of $X$. 
\end{Proposition}
\begin{proof}
First note that the cokernel of the graded linear map defined by the matrix $\Theta$ is  $\Q=R/J=R/I_2(\Theta)$. Then $e=\mathrm{indeg}(\syz(\Theta))=\mathrm{indeg}(\syz(J))$. Since $\gcd(f_1,f_2,f_3)=1$, we have $\hht(J)\ge 2$.
If $J$ is a complete intersection, then necessarily $\hht(J)=3$, hence $\dim(R/J)=n-3$. 
Therefore Proposition~\ref{BourM-q} yields $\Bour(\Theta)=d^2$. Assume now that $J$ is an almost complete intersection. Then $\hht(J)=2$, so $\dim(R/J)=n-2$. In particular, the Hilbert polynomial of $R/J$ has degree $n-3$, and therefore $e_0(R/J)=0$ and $e_1(R/J)=\deg(R/J)$.
 The desired formula then follows from equation~\eqref{MainBourDegFormual}. 

The second statement follows from the fact that the Jacobian ideal of $f$ is equigenerated of codimension $2$ and hence the singular subscheme of $X$ consists of finitely many points. 
\end{proof} 
We now characterize the extremal (and small) values of the Bourbaki degree in terms of the homological behavior of $J$ and the geometry of the associated Bourbaki ideal. We define the Bourbaki degree of the ideal $J$ as $\Bour(J):=\Bour(\Theta)$. The Bourbaki degree $\Bour(J)$ can be viewed as a numerical measure of how far the three equigenerated ideal $J$ is from being perfect, and in particular from being a complete intersection. 
\begin{Theorem}\label[Theorem]{Bourequi}
Let $J=(f_1,f_2,f_3)\subset R=k[x_1,\ldots,x_n]$ be a three-equigenerated  ideal.
Let $\nu\in \syz(J)$ be a minimal syzygy of degree $e:=\indeg(\syz(J))$, and let $I_\nu\subset R$ be the associated Bourbaki ideal. Then:
\begin{enumerate}
    \item[{\rm (i)}] $\Bour(J)=0$ if and only if $J$ is a perfect ideal (i.e., $R/J$ is Cohen-Macaulay). 
    \item[{\rm (ii)}] $\Bour(J)=1$ if and only if  $I_\nu$ is a complete intersection of two linear forms. 
 \item[{\rm (iii)}] $\Bour(J)=2$ if and only if $I_\nu$ is exactly one of the following: 
 \begin{enumerate}
\item $I_\nu=\fp_1\cap \fp_2$, where $\fp_1,\fp_2$ are distinct height $2$ linear prime ideals.
\item $I_\nu=\fp$ is a height $2$ prime ideal with $\deg(R/\fp)=2$, and after a linear change of
coordinates one has $I_\nu=(\ell,q)$, where $\ell$ is a linear form and $q$ is a quadratic form.
\item $I_\nu$ is $\fp$-primary for some height $2$ linear prime $\fp$ and
$\length\big((R/I_\nu)_\fp\big)=2$.
\end{enumerate}
 \item[{\rm (iv)}] $\Bour(J)=d^2-1$ if and only if $e=d$ and $\deg(R/J)=1$. 
 \item[{\rm (v)}] $\Bour(J)=d^2$ if and only if $J$ is a complete intersection. 
\end{enumerate}
\end{Theorem}
\begin{proof}
By definition $\Bour(J)=\deg(R/I_{\nu}).$

(i) Assume first that $\Bour(J)=0$.  Then $\deg(R/I_\nu)=0$, so $I_\nu\simeq R$ and $M_\nu$ is free. By Theorem~\ref{Bourideal}(a), this implies that $\Theta$ is free, hence $\syz(\Theta)$ is free. Consequently, $\syz(J)$ is free. For an ideal of height $2$ generated by three forms, a free syzygy module means that the minimal free resolution of $R/J$ has length $2$. Thus $\projdim(R/J)=2$. Since $\dim R/J = n-2$, the Auslander–Buchsbaum formula gives
$ \depth(R/J) = n-2 = \dim(R/J)$, therefore $R/J$ is Cohen–Macaulay, i.e. $J$ is perfect.

Conversely, suppose $J$ is perfect. Then $R/J$ has projective dimension $2$, and its minimal free resolution is the Hilbert–Burch resolution. 
Hence $\syz(J)$ is free, and thus $\syz(\Theta)$ is free. By Theorem~\ref{Bourideal}(a), $M_\nu$ is free, so $I_\nu \cong R$ and $\deg(R/I_\nu)=0$. Therefore $\Bour(J)=0$.

\medskip

(ii) If $\Bour(J)=1$, then $\deg(R/I_\nu)=1$. Since $I_\nu$ is unmixed of height $2$, the associativity formula for multiplicities yields
\[
\deg(R/I_\nu)=\sum_{\fp\in \Min(R/I_\nu)} \length\big((R/I_\nu)_\fp\big)\, \deg(R/\fp).
\]
Each summand is a positive integer. Hence $\deg(R/I_\nu)=1$ implies that $I_\nu$ has a unique minimal prime $\fp$ and $\length((R/I_\nu)_\fp)=1$. Therefore, $I_\nu$ is generically reduced, and since $I_\nu$ is unmixed, it follows that $I_\nu=\fp$ is prime.

Let $\fm=(x_1,\ldots,x_n)$. Then $A:=(R/I_\nu)_\fm$ is a local domain with $\deg(A)=1$, hence $A$
is a regular local ring. Since $R_\fm$ is regular and $A\simeq R_\fm/(I_\nu)_\fm$ has codimension $2$, the ideal $(I_\nu)_\fm$ is generated by a regular sequence of length $2$. Thus $I_\nu=(f,g)$ for some homogeneous elements  $f,g\in R$ forming a regular sequence. Consequently,
\[
1=\deg(R/I_\nu)=\deg(R/(f,g))=\deg(f)\deg(g),
\]
and therefore $\deg(f)=\deg(g)=1$. Hence $f$ and $g$ are linear forms, so $I_\nu=(\ell_1,\ell_2)$.  Conversely, if $I_\nu=(\ell_1,\ell_2)$ with $\deg\ell_i=1$, then $\deg(R/I_\nu)=1$, hence $\Bour(J)=1$.

\medskip

(iii) Assume that $\Bour(J)=2$. Then $\deg(R/I_\nu)=2$. 
As $I_\nu$ is unmixed, every minimal prime $\fp$ of $I_\nu$ has height $2$. By the associativity formula for multiplicities,
\[
2=\deg(R/I_\nu)=\sum_{\fp\in\Min(R/I_\nu)}\length\big((R/I_\nu)_\fp\big)\, \deg(R/\fp),
\]
and each summand is a positive integer. Hence, either:
\begin{align*}
\Min(R/I_\nu)&=\{\fp_1,\fp_2\}, \text{ with } \length((R/I_\nu)_{\fp_i})=\deg(R/\fp_i)=1,\\
\Min(R/I_\nu)&=\{\fp\}, \hspace{0.7cm} \text{ with } \hspace{0.35cm} \length\big((R/I_\nu)_\fp\big)\cdot \deg(R/\fp) = 2.
\end{align*}

In the first case, $\deg(R/\fp_i)=1$ implies that $\fp_i$ is generated by two independent linear forms. Moreover, $\length((R/I_\nu)_{\fp_i})=1$ implies that $(I_\nu)_{\fp_i}=\fp_iR_{\fp_i}$
Let $I_\nu=Q_1\cap Q_2$ be a minimal primary decomposition, where $Q_i$ is $\fp_i$-primary. Then localizing the primary decomposition at $\fp_i$ kills the other component, so $(I_\nu)_{\fp_i}=(Q_i)_{\fp_i}.$
Therefore
\[
(Q_i)_{\fp_i}=\fp_iR_{\fp_i}.
\]
We claim that $Q_i=\fp_i$. Indeed, if $Q_i\neq \fp_i$, choose
$a\in \fp_i\setminus Q_i$. Since $a/1\in \fp_iR_{\fp_i}=(Q_i)_{\fp_i}$, there exists $s\notin \fp_i$ such that
$sa\in Q_i.$ But $Q_i$ is $\fp_i$-primary and $s\notin \fp_i=\sqrt{Q_i}$, so primaryness implies $a\in Q_i$,
a contradiction. Thus $Q_i=\fp_i$, and therefore
 $I_\nu=\fp_1\cap\fp_2$. 
 
 In the second case, either $\deg(R/\fp)=2$ and $\length((R/I_\nu)_\fp)=1$, or $\deg(R/\fp)=1$ and $\length((R/I_\nu)_\fp)=2$. If $\deg(R/\fp)=1$, then $\fp$ is a height $2$ linear prime and $I_\nu$ is $\fp$-primary with $\length((R/I_\nu)_\fp)=2$. If $\deg(R/\fp)=2$, then $\fp$ is a height $2$ prime of degree $2$, hence $\fp$ is degenerate and contains a linear form $\ell$. Thus $\fp=(\ell,q)$ for some quadric $q$, and therefore the same primaryness argument used earlier in the proof then gives $I_\nu=\fp=(\ell,q)$. Conversely, it is immediate from the associativity formula for degrees that each of the cases (a), (b), and (c) yields that $\deg(R/I_\nu)=\Bour(J)=2$. 

\medskip

(iv) Suppose that $\Bour(J)=d^2-1$. Then $e(e-d)=\deg(R/J)-1$. The left side is non‑positive because $e\leq d$. Hence $\deg(R/J)\leq 1$. But $\deg(R/J)\geq 1$. Thus $\deg(R/J)=1$. Then $e(e-d)=0$, , so either $e=0$ or $e=d$. If $e=0$, then there is a constant syzygy, which would imply that the three generators are linearly dependent over $k$. Then the ideal is actually generated by two forms, making it a complete intersection. But a complete intersection has $\Bour(J)=d^2$ by \Cref{BourM-q}, a contradiction. Therefore $e=d$. The converse is clear.

\medskip

(v) Suppose that $\Bour(J)=d^2$ and $J$ is not a complete intersection. Then $\hht(J)=2$ since it cannot be height $1$ because of the gcd condition, and by the Bourbaki degree formula, one has $\deg(R/J)=e(e-d)$. Since $e\leq d$, it follows that $\deg(R/J)\leq 0$ which is impossible with $\dim(R/J)=n-2\geq 1$. The converse follows from \Cref{BourM-q}. 
\end{proof}

The following result is the analogue of \cite[Theorem~2.10]{MAA} for three--equigenerated ideals. It shows that, under the natural local freeness hypothesis on the syzygy module, the Bourbaki degree is bounded above by the square of the initial degree of syzygies.
\begin{Theorem}\label[Theorem]{thm:equigenerated-Bour-bound}
Let $J=(f_1,f_2,f_3)\subset R=k[x_1,\ldots,x_n]$ be a three-equigenerated  ideal and $e=\indeg{(\syz(J))}$. Assume that   $\syz(J)$ is locally free on the punctured spectrum. Then 
\[\Bour(J)\leq e^2. \]
\end{Theorem}
\begin{proof}
Let $\nu=(a_1,a_2,a_3)\in\syz(J)$ be a minimal syzygy of degree $e$, and let $I_\nu$ be the Bourbaki
ideal defined by the exact sequence
\[
0\lar R(-e)\xrightarrow{\ \nu\ }\syz(J)\lar I_\nu(e-d)\lar 0.
\]
Set $H:=(a_1,a_2,a_3)\subset R$. Dualizing the bottom row of Diagram~\eqref{diag1} yields an exact sequence
\begin{equation}\label{exact1}
 \Ext^1_R\!\big(R^3/\widetilde\nu(R(-e)),R\big)
\lar
\Ext^1_R\!\big(I_\nu(e-d),R\big)
\lar
\Ext^2_R(J,R)\lar 0.   
\end{equation}
From the free resolution
\[
0\lar R(-e)\xrightarrow{\ \widetilde\nu\ }R^3\lar R^3/\widetilde\nu(R(-e))\lar 0
\]
we obtain
\begin{equation}\label{isom1}
\Ext^1_R\!\big(R^3/\widetilde\nu(R(-e)),R\big)\ \cong\ (R/H)(e).
\end{equation}
Fix $\fp\in\Min(R/I_\nu)$. Since $\hht(I_\nu)=2$, we have $\fp\neq\fm$. By hypothesis $\syz(J)_{\fp}$ is
free, hence localizing $0\lar \syz(J)\lar R^3(-d)\lar J\lar 0$ gives $\projdim_{R_{\fp}}(J_{\fp})\le 1$,
so $\Ext^2_R(J,R)_{\fp}=0$. Localizing~\eqref{exact1} at $\fp$ and using~\eqref{isom1} yields a surjection
\begin{equation}\label{surjc1}
  (R/H)_{\fp}\twoheadrightarrow \Ext^1_R\!\big(I_\nu(e-d),R\big)_{\fp}.  
\end{equation}

Moreover, localizing the Bourbaki sequence at $\fp$ gives
$0\lar R_{\fp}\lar \syz(J)_{\fp}\lar (I_\nu)_{\fp}\lar 0$, hence $\projdim_{R_{\fp}}((I_\nu)_{\fp})\le 1$,
and therefore $\Ext^1_R(I_\nu(e-d),R)_{\fp}\cong (R/I_\nu)_{\fp}$ (up to shift). Taking lengths in~\eqref{surjc1} for all $\fp\in\Min(R/I_\nu)$
 we obtain
\[
\length\big((R/I_\nu)_{\fp}\big)\le \length\big((R/H)_{\fp}\big).
\]
Since $I_\nu$ is unmixed, associativity of multiplicity yields $\deg(R/I_\nu)\le \deg(R/H)$. Finally,
$H$ is generated by three forms of degree $e$, hence $\deg(R/H)\le e^2$. Therefore
$\Bour(J)=\deg(R/I_\nu)\le e^2$.
\end{proof}

Although the argument above works only for three-equigenerated ideals $J$ with the extra-hypothesis of local freeness, the lack of examples motivates the following question.

\begin{Question}\label[Question]{quest:Bour-e^2-bound}
Does the bound $\Bour(\Theta) \leq e^2$ holds in general for any $2 \times 4$-matrix $\Theta$ with rank $2$?
\end{Question}

A related question is posed after \Cref{prop:initial-degree-1-n-4}, where we show that $\Bour(\Theta)\leq 2$ whenever $e = 1$. An example showing this bound is sharp would give a negative answer to the question above. 

\begin{Corollary}
Let $J=(f_1,f_2,f_3)\subset R=k[x_1,\ldots,x_n]$ be a three-equigenerated  ideal and $e=\indeg{(\syz(J))}$.  Assume that  $\syz(\Theta)$ is locally free on the punctured spectrum. Then 
\[d(d-e)\le \deg(R/J)\le d^2+e^2-ed.\]
\end{Corollary}
\begin{proof}
Follows from $0\leq \Bour(J)\leq e^2$ and the formula~\eqref{bourthetaJ}.
\end{proof}

\begin{Remark}\label[Remark]{rmk:duPlessis-Wall}
The inequalities above, for $n=3$, generalize the ones obtained for Jacobian ideals in \cite{CTC-Plessis}.
\end{Remark}

\begin{Example}\label[Example]{ex:quadrics-Bour2}
Let us consider the ideal $J = (x_1 x_4, x_2 x_3, x_1 x_3 - x_2 x_4)$. The associated matrix will be
$$
\Theta = \begin{bmatrix}
x_1 x_4 & x_2 x_3 & x_1 x_3 - x_2 x_4 & 0\\
0       &    0    &          0        & 1
\end{bmatrix}.
$$
Here, the degree of a minimal syzygy is $e = 2$, $d_1 = 2$ and $d_2 = 0$, and
$$
I_2(\Theta) = (x_1 x_4, x_2 x_3, x_1 x_3 - x_2 x_4) = J.
$$
The associated primes to this ideal are $(x_3, x_4)$, $(x_1, x_2)$ and $(x_1, x_2, x_3, x_4)$, so $e_0(\Theta) = 0$ and $\Theta$ is locally free (from \Cref{free-locallyfree}). We may also compute the lengths in each minimal prime, to obtain $e_1(\Q) = 2$ from the associativity formula for the cokernel.

Using the Bourbaki degree formula, we obtain $\Bour(\Theta) = \Bour(J) = 2$ for every $n \geq 4$. This value for the Bourbaki degree cannot occur for $n = 3$ and Jacobian ideals of cubic plane curves, as was observed in \cite[Section $5.1$]{Futata2023} via a case-by-case analysis.

The minimal free resolution is of the form:
$$
0 \to R(-4) \to R^4(-3) \to R^5(-2) \to R^4 \xrightarrow{\Theta} R(2) \oplus R.
$$
For $n = 4$, this is an example of a matrix which induces a null-correlation distribution of degree $2$ in $\mathbb{P}^3$, as we will discuss in \Cref{subsec:syz-dist}.
\end{Example}

\begin{Theorem}\label[Theorem]{prop:loc-free-d-2-Bour2-eq-ideal}
Let $J=(f_1,f_2,f_3)\subset R=k[x_1,\ldots,x_n]$ be a three-equigenerated, height two ideal with $\deg(f_i) = 2$. Then:
\begin{itemize}
    \item[(a)] If $n = 3$, then $\Bour(J) \neq 2$;
    \item[(b)] If $n \geq 4$,  $\Theta$ is locally free and $J$ is saturated, then $\Bour(J) \neq 2$.
\end{itemize}
\end{Theorem}
\begin{proof}
From the local freeness assumption, which holds for $n=3$ without further hypothesis, we may use \Cref{thm:equigenerated-Bour-bound} to conclude that $\Bour(J) = 2 \Rightarrow e = 2$, since $e \leq d = 2$. 

For $(b)$, assuming that $\height(J) = 2$, we first note that the projective dimension of $R/J$ is either $2$ or $3$. From the Hilbert syzygy theorem, a general minimal free resolution for $R/J$ is of the form
$$
0 \to F_3 \to F_2 \to F_1 \to F_0 \to R/J \to 0,
$$
and the projective dimension is at most $3 = \dim(R)$. Since $\height(J) = 2$, $\dim(R/J) = 1$, and we know in general that $\depth(R/J) \leq \dim(R/J) = 1$. From the Auslander-Buchsbaum formula
$$
\text{proj.dim}(R/J) = \dim(R) - \depth(R/J),
$$
we obtain that $\text{proj.dim}(R/J) \geq 2$. We note moreover that, for $\height(J) = 2$, we have the equivalence
$$
\text{proj. dim}(R/J) = 2 \iff J \text{ is saturated.}
$$
Indeed, the maximal ideal $\fm=(x_1, x_2, x_3) \in \Ass(R/J)$ if and only if $\depth(R/J) = 0$, which corresponds to $\text{proj.dim.}(R/J) = 3$. A minimal free resolution for $R/J$ will be of the form
\begin{align}\label{eq:non-degenerate-res-d-2}
0 \to \bigoplus_{k=1}^{r-2} R(-4-e_k) \to \bigoplus_{i=1}^{r} R(-4) \to R(-2)^3 \to R \to R/J \to 0,     
\end{align}
since $\deg(f_i) = 2$ and $e = 2$ is the only possible degree for syzygies of $J$, where $r \geq 2$ and $e_k \geq 1$. This is a special case of a \emph{non-degenerate ideal}, a concept introduced in \cite[Definition 2.1]{Simis1977}, where the particular case of generators with $\deg(f_i) = 2$ follows from \cite[Proposition 6]{Lazard1977}.

Note that $r = 2$ if and only if the resolution is Hilbert-Burch, and this corresponds to the case where $J$ is perfect of height two, contradicting $\Bour(J) = 2 \neq 0$. Hence, $r \geq 3$. Using the additivity of the Hilbert polynomial in the sequence, we obtain
\begin{align*}
2 &= \sum_{i=1}^{r} 2 - \sum_{k=1}^{r-2} (2-e_{k})\\
&= 2r - 2r + 4 - \sum_{k=1}^{r-2} e_k
\end{align*}
showing that $\sum_{k=1}^{r-2} e_k = 2$. Together with $e_k \geq 1$, we obtain two possibilities, either $r = 4$ and $e_1=e_2 = 1$ or $r = 3$ and $e_1 = 2$. These correspond, via \Cref{Thm:Bourbaki-resolutions}, to the following free resolutions for the Bourbaki ideal $I_\nu$, after a choice of syzygy:
$$
0 \to R^2(-4) \to R^3(-2) \to I_\nu \to 0
$$
or
$$
0 \to R(-4) \to R^2(-2) \to I_\nu \to 0.
$$
The first case gives $\deg(R/I_\nu) = 3$, and the second gives $\deg(R/I_\nu) = 4$, from comparing the Hilbert polynomials, so both cases are impossible.

To show $(b)$, let $n \geq 4$, let us assume that $J$ is saturated. Then $\height(J) = 2$ and, from \Cref{lem:loc-free-associated-primes}, we conclude that the ideal $J$ is unmixed, hence a locally Cohen-Macaulay ideal. From the Bourbaki degree formula with $\Bour(J) = 2$ and $e = 2$, we obtain that $\deg(R/J) = e_1(\Q) = 2$, and by the associativity formula
$$
\deg(R/J) = \sum_{\fp \in \Ass{R/J}} \length(R_\fp/J_\fp)\deg(R/\fp)
$$
we have three possibilities:
\begin{itemize}
    \item[(a)] There is only one prime $\fp \in \Ass(R/J)$, so that $\length(R_\fp/J_\fp) = 2$ and $\deg(R/\fp) = 1$. But the condition on the degree implies that $\fp$ is the ideal of an $(n-2)$-hyperplane, given by two linear forms $\fp = (l_1, l_2)$. Now, we divide the proof into two cases, $n = 4$ and $n \geq 5$. For $n=4$, we are working over $\mathbb{P}^3$, and the scheme $R/J$ corresponds to a multiplicity two structure on a line. Then, from \cite[Proposition 1.4]{nollet1997hilbert}, since $J$ is locally Cohen-Macaulay, we obtain the following general form of the ideal:
    $$
    J = (x_1^2, x_1 x_2, x_2^2, x_1 g - x_2 f)
    $$
    where $f,g$ are homogeneous polynomials of the same degree without common zeros. Thus, this contradicts the hypothesis on the number of generators for $J$. For $n \geq 5$, we may assume that $\fp =(x_1, x_2)$, up to a change of coordinates, and consider the ideal $I = (x_3, \ldots, x_n)$. The associated closed locus $V(I) \simeq \mathbb{P}^3 \subset \mathbb{P}^{n-1}$, so that the intersection $V(J) \cap V(I) = V(J') \subset \mathbb{P}^3$ is a double structure on a line, and using the previous argument we obtain an ideal of the same form:
    $$
    J' = (x_1^2, x_1 x_2, x_2^2, x_1 g - x_2 f).
    $$
    Coming back to $n$-variables, since $J' = J + I$, we also need more than $3$ generators for $J$.
    
    \item[(b)] There are two associated primes $\fp_1, \fp_2$, so that $\length(R_{\fp_i}
    /J_{\fp_i}) = 1$ and $\deg(R/\fp_i) = 1$, so two ideals generated by two linear forms each.  Moreover, exactly as in the proof of Theorem~\ref{Bourequi}(iii), one obtains  $J = \fp_1 \cap \fp_2$. It follows that $J$ has four minimal generators, contradicting the hypothesis.
    \item[(c)] There is only one associated prime $\fp$, with $\length(R_\fp/J_\fp) = 1$ and $\deg(R/\fp) = 2$. Then $\fp$ is a height $2$ prime of degree $2$, hence $\fp$ is degenerate and contains a linear form $l$, with $\fp = (l,q)$ for some quadric polynomial $q \in R_2$.  Again, by the same argument as in the proof of Theorem~\ref{Bourequi}(iii), it follows that
 $J = \fp = (l,q)$, which contradicts the hypothesis on the generators of $J$.
\end{itemize}
Hence, all possibilities lead to a contradiction. If we remove the hypothesis of $J$ being saturated for $n \geq 4$, then the case $(b)$ may happen, as in \Cref{ex:quadrics-Bour2} above, where $J \neq \fp_1 \cap \fp_2$, but $I_\nu = \fp_1 \cap \fp_2$ is the intersection of two height $2$ linear primes.
\end{proof}

\begin{Example}
We include an example where the resolution cannot be as in \Cref{eq:non-degenerate-res-d-2}, where $J=(f_1, f_2, f_3)$ is three-equigenerated with $\deg(f_i) \geq 3$. The concept of \emph{non-degenerate} ideals is related to the syzygies of $\syz(J)$ being of degree at most $d$. This is not the case for example for irreducible plane curves with one node, say $f = xyz^{d-1} + x^{d+1} + y^{d+1}$ for $d \geq 4$. In particular, if $d = 4$, the minimal free resolution for the associated Jacobian ideal $J_f$ is
$$
0 \to R(-11)^2 \to R(-10) \oplus R(-8)^3 \to R(-4)^3 \to R \to R/J \to 0,
$$
where one of the generating syzygies has degree $6 \geq 4 = d$. In general, if the ideal $J$ is non-degenerate and the initial degree $e=d$, then the general shape of the resolution will be of the form
$$
0 \to \bigoplus_{k=1}^{r-2} R(-2d-e_k) \to \bigoplus_{i=1}^r R(-2d) \to R(-d)^3 \to R \to R/J \to 0.
$$
\end{Example}


\section{The Bourbaki degree of a linear matrix}\label{sec:linear-matrices}

In this section, we investigate the Bourbaki degree of a $2\times 4$ matrix of linear forms in $R=k[x_1,\ldots,x_n]$. We start by recalling the Kronecker--Weierstrass
normal form for matrices of linear forms, which provides a complete
classification up to equivalence.

Let $\Theta$ be a $2\times n$ matrix of linear forms in $R$. Two such matrices are
said to be \emph{equivalent} if they differ by left multiplication by an element of
$\GL_2(k)$ and right multiplication by an element of $\GL_n(k)$. The classical
Kronecker--Weierstrass theorem asserts that any $2\times n$ matrix of linear
forms is equivalent to a block matrix obtained by the concatenation of three
types of blocks: nilpotent blocks, Jordan blocks, and scroll blocks.

\medskip

\noindent
\textbf{Nilpotent blocks.}
A nilpotent block of length $m+1$ has the form
\[
D_m=
\begin{bmatrix}
	x_1 & x_2 & \cdots & x_m & 0\\
	0   & x_1 & \cdots & x_{m-1} & x_m
\end{bmatrix}.
\]

\noindent
\textbf{Jordan blocks.}
A Jordan block of length $m$ with eigenvalue $\lambda\in k$ has the form
\[
J_m(\lambda)=
\begin{bmatrix}
	y_1 & y_2 & \cdots & y_m\\
	\lambda y_1 & y_1+\lambda y_2 & \cdots & y_{m-1}+\lambda y_m
\end{bmatrix}.
\]

\noindent
\textbf{Scroll blocks.}
A scroll block of length $m$ has the form
\[
B_m=
\begin{bmatrix}
	z_1 & z_2 & \cdots & z_m\\
	z_0 & z_1 & \cdots & z_{m-1}
\end{bmatrix}.
\]

Here, all variables appearing in different blocks are assumed to be distinct, algebraically independent
linear forms in $R$. The Kronecker--Weierstrass theorem states that $\Theta$ is
equivalent to a block matrix obtained by concatenating finitely many blocks of the above types, uniquely determined up to permutation of blocks.

\medskip

We now specialize this classification to the case of $2\times4$ matrices.

\begin{Proposition}\label[Proposition]{KWform}
	Let $\Theta$ be a $2\times4$ matrix of linear forms in $R$. Then $\Theta$ is equivalent to exactly one of the following matrices.
	\begin{itemize}
		\item[{\rm (i)}] One single block:
		\[
		B_4=
		\begin{bmatrix}
			x_2 & x_3 & x_4 & x_5\\
			x_1 & x_2 & x_3 & x_4
		\end{bmatrix}
		,\quad  D_3=\begin{bmatrix}
			x_1 & x_2 & x_3 & 0\\
			0   & x_1 & x_2 & x_3
		\end{bmatrix}\]
		or \[
		 J_4(\lambda)=\begin{bmatrix}
		x_1 & x_2 & x_3 & x_4\\
		\lambda x_1 & x_1+\lambda x_2 & x_2+\lambda x_3 & x_3+\lambda x_4
		\end{bmatrix}.
		\]
		
		\item[{\rm (ii)}] Two-block decomposition:
    
{\small
\begin{align*}
D_2 \mid B_1 &=
\begin{bmatrix}
x_1 & x_2 & 0 & x_3\\
0   & x_1 & x_2 & x_4
\end{bmatrix},
&
D_2 \mid J_1(\lambda) &=
\begin{bmatrix}
x_1 & x_2 & 0 & x_3\\
0   & x_1 & x_2 & \lambda x_3
\end{bmatrix},
\\[6pt]
J_3(\lambda) \mid B_1 &=
\begin{bmatrix}
x_1 & x_2 & x_3 & x_4\\
\lambda x_1 & x_1+\lambda x_2 & x_2+\lambda x_3 & x_5
\end{bmatrix},
&
B_3 \mid J_1(\mu) &=
\begin{bmatrix}
x_1 & x_2 & x_3 & x_5\\
x_2 & x_3 & x_4 & \mu x_5
\end{bmatrix},
\\[6pt]
B_3 \mid B_1 &=
\begin{bmatrix}
x_1 & x_2 & x_3 & x_5\\
x_2 & x_3 & x_4 & x_6
\end{bmatrix},
&
B_2 \mid B_2 &=
\begin{bmatrix}
x_2 & x_3 & x_5 & x_6\\
x_1 & x_2 & x_4 & x_5
\end{bmatrix},
\\[6pt]
D_1 \mid J_2(\lambda) &=
\begin{bmatrix}
x_1 & 0 & x_2 & x_3\\
0   & x_1 & \lambda x_2 & x_2+\lambda x_3
\end{bmatrix},
&
D_1 \mid B_2 &=
\begin{bmatrix}
x_1 & 0 & x_3 & x_4\\
0   & x_1 & x_2 & x_3
\end{bmatrix},
\\[6pt]
J_2(\lambda) \mid J_2(\mu) &=
\begin{bmatrix}
x_1 & x_2 & x_3 & x_4\\
\lambda x_1 & x_1+\lambda x_2 & \mu x_3 & x_3+\mu x_4
\end{bmatrix},
&
D_1 \mid D_1 &=
\begin{bmatrix}
x_1 & 0 & x_2 & 0\\
0   & x_1 & 0   & x_2
\end{bmatrix},
\\[6pt]
J_3(\lambda) \mid J_1(\mu) &=
\begin{bmatrix}
x_1 & x_2 & x_3 & x_4\\
\lambda x_1 & x_1+\lambda x_2 & x_2+\lambda x_3 & \mu x_4
\end{bmatrix},
&
J_2(\lambda) \mid B_2 &=
\begin{bmatrix}
x_1 & x_2 & x_4 & x_5\\
\lambda x_1 & x_1+\lambda x_2 & x_3 & x_4
\end{bmatrix}.
\end{align*}
}

\item[{\rm (iii)}] Three-block decomposition:
 \[
B_2 \mid B_1 \mid B_1 =
\begin{bmatrix}
x_2 & x_3 & x_5 & x_6\\
x_1 & x_2 & x_4 & x_5
\end{bmatrix},\quad 
J_2(\lambda) \mid B_1 \mid J_1(\mu) =
\begin{bmatrix}
x_1 & x_2 & x_4 & x_5\\
\lambda x_1 & x_1+\lambda x_2 & x_3 & \mu x_5
\end{bmatrix},
\]

\[
B_2 \mid J_1(\mu) \mid J_1(\lambda) =
\begin{bmatrix}
x_2 & x_3 & x_4 & x_5\\
x_1 & x_2 & \mu x_4 & \lambda x_5
\end{bmatrix}.\quad 
D_1 \mid B_1 \mid B_1 =
\begin{bmatrix}
x_1 & 0 & x_3 & x_5\\
0   & x_1 & x_2 & x_4
\end{bmatrix},
\]

\[
D_1 \mid B_1 \mid J_1(\lambda ) =
\begin{bmatrix}
x_1 & 0 & x_3 & x_4\\
0   & x_1 & x_2 & \lambda x_4
\end{bmatrix},\quad 
D_1 \mid J_1(\mu) \mid J_1(\lambda) =
\begin{bmatrix}
x_1 & 0 & x_2 & x_3\\
0   & x_1 & \mu x_2 & \lambda x_3
\end{bmatrix},
\]

\[
J_2(\lambda) \mid B_1 \mid B_1 =
\begin{bmatrix}
x_1 & x_2 & x_4 & x_6\\
\lambda x_1 & x_1+\lambda x_2 & x_3 & x_5
\end{bmatrix},\ \  
B_2 \mid B_1 \mid J_1(\mu) =
\begin{bmatrix}
x_2 & x_3 & x_5 & x_6\\
x_1 & x_2 & x_4 & \mu x_6
\end{bmatrix},
\]

\[
J_2(\lambda) \mid J_1(\mu) \mid J_1(\rho ) =
\begin{bmatrix}
x_1 & x_2 & x_3 & x_4\\
\lambda x_1 & x_1+\lambda x_2 & \mu x_3 & \rho x_4
\end{bmatrix}\]
				
\item[{\rm (iv)}]  Four-block decomposition:
\[
J_1(\lambda_1)\mid J_1(\lambda_2)\mid J_1(\lambda_3)\mid J_1(\lambda_4)=\begin{bmatrix}
x_1&x_2&x_3&x_4\\
\lambda_1x_1&\lambda_2x_2&\lambda_3x_3&\lambda_4x_4
\end{bmatrix},
\]
\[ B_1\mid B_1\mid B_1\mid B_1=
\begin{bmatrix}
x_2&x_4&x_6&x_8\\
x_1&x_3&x_5&x_7
\end{bmatrix}
\]
\end{itemize}
\end{Proposition}

We examine each of the above cases, describing the Bourbaki degree, its minimal resolution, and other discrete invariants.

\begin{Proposition}\label[Proposition]{KWform-BR-resolution}
Let the Kroenecker-Weierstrass form of a $2 \times 4$ linear matrix $\Theta$ be one of the following:
\begin{align*}
&B_4,\,  D_3,\,  J_4(\lambda),\,  J_3(\lambda) \mid B_1, B_3 \mid J_1(\mu),\,  B_3 \mid B_1, B_2 \mid B_2,\, J_2(\lambda) \mid B_2\\
&B_2 \mid B_1 \mid B_1,\,  B_2 \mid B_1 \mid J_1(\lambda),\,  J_2(\lambda) \mid B_1 \mid B_1,\,  B_1 \mid B_1 \mid B_1 \mid B_1,
\end{align*}
or
\begin{align*}
J_3(\lambda) \mid J_1(\mu),\,  J_2(\lambda) \mid J_2(\mu),\,  B_2 \mid J_1(\lambda) \mid J_1(\mu),\,  J_2(\lambda) \mid B_1 \mid J_1(\mu),\,  \ \  & \text{with }\lambda \neq \mu, \\ 
J_2(\lambda) \mid J_1(\mu) \mid J_1(\rho),\,  \ \  & \text{with }\lambda, \mu, \rho \text{ distinct},\\
J_1(\lambda_1) \mid J_1(\lambda_2) \mid J_1(\lambda_3) \mid J_1(\lambda_4) \ \ &  \text{with }\lambda_1, \ldots, \lambda_4 \text{ distinct}.
\end{align*}
Then $\hht I_2(\Theta) =3$, and in particular we have $\Bour(\Theta) = 3$, where the minimal free resolution for $\Theta$ is of Buchsbaum--Rim form, namely
$$
0 \to R^2(-3) \to R^4(-2) \to R^4 \xrightarrow{\Theta} R^2(1).
$$
\end{Proposition}

\begin{proof}
The claim about the height of the ideal $I_2(\Theta)$ follows from the general formula in \cite[Proposition 2.2]{Abbas-AluffiTFIdeals}, which describes this height in terms of the types of each block involved in the Kronecker-Weierstrass decomposition. Then, the minimal free resolution is of Buchsbaum--Rim type since $\dim \Q \leq n-3$, from \Cref{dimQ-n-3-BR-resolution}, and moreover
$$
\Bour(\Theta) = \fq_{\Theta} = d_1^2 + d_2^2 + d_1 d_2 =  3
$$
follows from \Cref{BourM-q}.
\end{proof}

\begin{Proposition}\label[Proposition]{KWform-free}
Let the Kroenecker-Weierstrass form of a $2 \times 4$ linear matrix $\Theta$ be one of the following:
\begin{align*}
D_1 \mid J_2(\lambda), D_1 \mid B_2, J_2(\lambda) \mid J_2(\lambda), D_1 \mid D_1, D_1 \mid B_1 \mid B_1, D_1 \mid B_1 \mid J_1(\lambda),
\end{align*}
or
\begin{align*}
D_1 \mid J_1(\mu) \mid J_1(\lambda), \ \  & \text{with }\lambda \neq \mu, \lambda \neq 0 \text{ or } \mu \neq 0, \\ 
J_1(\lambda) \mid J_1(\lambda) \mid J_1(\mu) \mid J_1(\mu), \ \  & \text{with }\lambda \neq \mu.
\end{align*}
Then $\Theta$ is free, so that $\syz(\Theta) \simeq R^2(-1)$.
\end{Proposition}

\begin{proof}
To show $\syz(\Theta) \simeq R^2(-1)$, it suffices to exhibit two linearly independent syzygies that generate the module and have no relations. For example, if 
$$
\Theta = D_1 \mid J_2(\lambda) = \begin{bmatrix}
x_1 & 0 & x_2 & x_3\\
0   & x_1 & \lambda x_2 & x_2+\lambda x_3
\end{bmatrix},
$$
then the matrix
$$
\begin{bmatrix}
-x_2 & -x_3\\
-\lambda x_2 & -x_2 - \lambda x_3\\
x_1 & 0\\
0 & x_1
\end{bmatrix}
$$
clearly defines an isomorphism $\syz(\Theta) \simeq R^2(-1)$. For the other cases, one can build such explicit syzygies using an algebra computer software, for example, Macaulay2 (\cite{M2}).
\end{proof}

\begin{Proposition}\label[Proposition]{KWform-nearlyfree-codimQ-2}
Let the Kroenecker-Weierstrass form of a $2 \times 4$ linear matrix $\Theta$ be one of the following:
\begin{align*}
D_2 \mid J_1(\lambda), J_3(\lambda) \mid J_1(\lambda), B_2 \mid J_1(0) \mid J_1(0), J_2(\lambda) \mid B_1 \mid J_1(\lambda),
\end{align*}
or
\begin{align*}
J_2(\lambda) \mid J_1(\mu) \mid J_1(\mu), J_2(\lambda) \mid J_1(\lambda) \mid J_1(\lambda), J_1(\lambda) \mid J_1(\lambda) \mid J_1(\mu) \mid J_1(\rho), 
\end{align*}
with $\lambda \neq \mu, \rho \neq \lambda, \mu$.
Then $\Theta$ is nearly free, with $\codim(\Q) = 2$, so that $e_0(\Q) = 0$, $e_1(\Q) = 1$, the degree of a minimal syzygy is $e = 1$ and $\Bour(\Theta) = 1$. The minimal free resolution for $\Theta$ is of the form
$$
0 \to R(-3) \to R^2(-2) \oplus R(-1) \to R^4 \xrightarrow{\Theta} R^2(1).
$$
\end{Proposition}

\begin{proof}
For each matrix listed above, we will show that $\codim(\Q) = 2$ (so $e_0(\Q) = 0$), the existence of a syzygy of degree $1$ (so $e=1$) and that $e_1(\Q) = 1$, using the associativity formula. Then, the claim follows from the Bourbaki degree formula:
$$
\Bour(\Theta) = (e-d)(e) + \fq_{\Theta} - e_1(\Q) = 1. 
$$
In particular, if $\nu \in \syz(\Theta)_1$ is a syzygy of minimum degree, $\deg(R/I_\nu) = 1$ and we obtain the claimed resolution using \Cref{prop:nearly-free-3-syzygy}, since $s = e-d+e_0(\Q) = -1$. 

Let $\Theta = D_2 \mid J_1(\lambda) = \begin{bmatrix}
x_1 & x_2 & 0 & x_3\\
0   & x_1 & x_2 & \lambda x_3
\end{bmatrix}$. In this case,
$$
I_2(\Theta) = (x_1^2, x_2^2, x_1 x_2, \lambda x_1 x_3, x_3(\lambda x_2 - x_1), - x_2 x_3),
$$
and the only prime ideals in its primary decomposition are $\fp = (x_1, x_2)$ and $(x_1, x_2, x_3)$, so $\codim(\Q) = 2$. The vector
$\nu = (-x_3, -\lambda x_3,-\lambda^2 x_3, x_1 + \lambda x_2)^T$ is a syzygy of degree $1$ for $\Theta$. 

From the associativity formula, we may compute
$$
e_1(\Q) = \length(\Q_{\fp}) \cdot \deg(R/\fp).
$$
Since $R/\fp \simeq k[x_3, \ldots, x_n]$, $\deg(R/\fp) = 1$. Localizing the sequence defining $\Q_{\fp}$ we obtain
$$
R_{\fp}^4 \xrightarrow{\Theta_{\fp}} R_{\fp}^2 \to \Q_{\fp} \to 0,
$$
where we may write
$$
\Theta_{\fp} = \begin{bmatrix}
x_1 & x_2 & 0 & x_3\\
0 & x_1 - \lambda x_2 & 0 & 0
\end{bmatrix}.
$$
The scaled fourth column vector $ue_4$ in $R_{\fp}^4$ goes to $(1,0)$, where $u = 1/x_3 \in R_{\fp}$. Writing $R_{\fp}^2 = (1,0)\cdot R_{\fp} \oplus (0,1) \cdot R_{\fp}$, we may compute the cokernel using the second row of the matrix, by
\begin{align*}
\cok(\Theta_{\fp}) \simeq R_{\fp}/( x_1, x_2) \simeq R_{\fp}/\fp_{\fp} \simeq k
\end{align*}
and thus $\length(\Q_{\fp}) = 1$, concluding $e_1(\Q) = 1$.

For $\Theta = J_3(\lambda) \mid J_1(\lambda) = \begin{bmatrix}
x_1 & x_2 & x_3 & x_4\\
\lambda x_1 & x_1+\lambda x_2 & x_2+\lambda x_3 & \lambda x_4
\end{bmatrix}$, the ideal of minors is given by
$$
I_2(\Theta) = (x_1^2, x_1 x_2, x_2^2 - x_1 x_3, - x_1 x_4, - x_2 x_4),
$$
and the prime ideals in its primary decomposition are $\fp = (x_1, x_2)$ and $(x_1, x_2, x_4)$, so $\codim(\Q) = 2$. The vector $\nu = (-x_4,0,0,x_1)^T$ is a syzygy of degree $1$ for $\Theta$. From the associativity formula, we may compute
$$
e_1(\Q) = \length(\Q_{\fp}) \cdot \deg(R/\fp).
$$
Since $R/\fp \simeq k[x_3, \ldots, x_n]$, we get $\deg(R/\fp) = 1$. On the other hand, localizing the sequence defining $\Q$ at $\fp$, we get
$$
R_{\fp}^4 \xrightarrow{\Theta_{\fp}} R_{\fp}^2 \to \Q_{\fp} \to 0,
$$
where we may rewrite, if $\lambda \neq 0$,
$$
\Theta_{\fp} = \begin{bmatrix}
x_1 & x_2 & x_3 & x_4\\
x_1 & \lambda^{-1}x_1 + x_2 & \lambda^{-1}x_2 + x_3 & x_4
\end{bmatrix} \sim \begin{bmatrix}
x_1 & x_2 & x_3 & x_4\\
0 & \lambda^{-1}x_1 & \lambda^{-1}x_2 & 0
\end{bmatrix}.
$$
Since $x_4 \in R_{\fp}$ is invertible, the matrix $\Theta_{\fp}$ sends the scaled fourth basis vector $ue_4$ from $R_{\fp}^4$ to $(1,0)$ in $R_{\fp}^2$, where $u = 1/x_4$, so that the remaining part is $(0,1) \cdot R_{\fp}$, and thus
$$
\cok(\Theta_{\fp}) \simeq R_{\fp}/( \lambda^{-1} x_1, \lambda^{-1}x_2 ) \simeq R_{\fp}/\fp_{\fp} \simeq k,
$$
concluding that $\length(\Q_{\fp}) = 1$ and $e_1(\Q) = 1$. On the other hand, if $\lambda = 0$, then 
$$
\Theta_{\fp} = \begin{bmatrix}
x_1 & x_2 & x_3 & x_4\\
0   & x_1 & x_2 & 0
\end{bmatrix}
$$
and using the same reasoning, we obtain
$$
\cok(\Theta_{\fp}) \simeq R_{\fp}/(x_1, x_2) \simeq R_{\fp}/\fp_{\fp} \simeq k,
$$
and therefore $e_1(\Q) = 1$.

Let $\Theta = B_2 \mid J_1(0) \mid J_1(0) = \begin{bmatrix}
x_2 & x_3 & x_4 & x_5\\
x_1 & x_2 & 0   & 0
\end{bmatrix}$. In this case, the ideal of minors is
$$
I_2(\Theta) = (x_2^2 - x_1 x_3, -x_1 x_4, - x_2 x_4, - x_1 x_5, - x_2 x_5),
$$
and the primes in its primary decomposition are $\fp = (x_1, x_2)$ and $(x_4, x_5, x_1^2 - x_1 x_3)$, so $\codim(\Q) = 2$. The vector $\nu = (0, 0, -x_5, x_4)^T$ is a syzygy of degree $1$ for $\Theta$.

From the associativity formula, we may compute 
$$
e_1(\Q) = \length(\Q_{\fp}) \cdot \deg(R/\fp).
$$
Since $R/\fp \simeq k[x_3, \ldots, x_n]$, $\deg(R/\fp)=1$. Now, localizing the sequence defining $\Q$ at $\fp$, we obtain
$$
R_{\fp}^4 \xrightarrow{\Theta_{\fp}} R_{\fp}^2 \to \Q_{\fp} \to 0,
$$
with
$$
\Theta_{\fp} = \begin{bmatrix}
x_2 & x_3 & x_4 & x_5\\
x_1 & x_2 & 0   & 0
\end{bmatrix}.
$$
Thus, the matrix $\Theta_{\fp}$ sends the scaled fourth basis vector $ue_4$ by $u = 1/x_5$ to $(1,0)$ inside $R_{\fp}^2$ and, since $R_{\fp}^2 = (1,0) \cdot R_{\fp} \oplus (0,1) \cdot R_{\fp}$, the cokernel of $\Theta_{\fp}$ can be computed by the quotient of the second row in the remaining matrix:
$$
\cok(\Theta_{\fp}) \simeq R_{\fp}/(x_1, x_2) = R_{\fp}/\fp_{\fp} \simeq k,
$$
and we conclude $\length(\Q_{\fp}) = 1$ and $e_1(\Q) = 1$.

Let $\Theta = J_2(\lambda) \mid B_1 \mid J_1(\lambda)= \begin{bmatrix}
x_1 & x_2 & x_4 & x_5\\
\lambda x_1 & x_1+\lambda x_2 & x_3 & \lambda x_5
\end{bmatrix}$. In this case, the ideal of minors is
$$
I_2(\Theta) = (x_1^2, x_1 (x_3 - \lambda x_4), x_5(\lambda x_4 - x_3), x_2 x_3 - x_4(x_1 + \lambda x_2), -x_1 x_2),
$$
and the primes in its primary decomposition are $\fp = (x_3 - \lambda x_4)$, $(x_1, x_2, x_5)$ and $(x_1, x_3 - \lambda x_4, x_5)$, so $\codim(\Q) = 2$. The vector $\nu = (0, 0, -x_5, x_4)^T$ is a syzygy of degree $1$ for $\Theta$.

From the associativity formula, we may compute $
e_1(\Q) = \length(\Q_{\fp}) \cdot \deg(R/\fp)$. Since $R/\fp \simeq k[x_2, x_4, x_5, \ldots, x_n]$, we have $\deg(R/\fp) = 1$. Localizing the sequence defining $\Q$ at $\fp$, we may write
$$
\Theta_{\fp} = \begin{bmatrix}
x_1 & x_2 & x_4               & x_5\\
0   & x_1 & x_3 - \lambda x_4 & 0 
\end{bmatrix}.
$$
The matrix $\Theta_{\fp}$ sends the scaled fourth basis vector $u e_4$ to $(1,0)$ in $R_{\fp}^2$, where $u = 1/x_5 \in R_{\fp}$. Since $R_{\fp}^2 = (1,0) \cdot R_{\fp} \oplus (0,1) \cdot R_{\fp}$, we may compute the cokernel of $\Theta_{\fp}$ by considering the second row of the remaining matrix, which gives
$$
\cok(\Theta_{\fp}) \simeq R_{\fp}/(x_1, x_3 - \lambda x_4) = R_{\fp}/\fp_{\fp} \simeq k,
$$
and thus $\length(\Q_{\fp}) = 1$ and $e_1(\Q) = 1$.

Let $\Theta = J_2(\lambda) \mid J_1(\mu) \mid J_1(\mu) = \begin{bmatrix}
x_1 & x_2 & x_3 & x_4\\
\lambda x_1 & x_1+\lambda x_2 & \mu x_3 & \mu x_4
\end{bmatrix}$, where $\lambda \neq \mu$. Here, the ideal of minors is given by
$$
I_2(\Theta) = (x_1^2, (\mu - \lambda)x_1 x_3, (\mu - \lambda)x_1 x_4, (\mu - \lambda)x_2 x_3 - x_1 x_3, (\mu - \lambda)x_2 x_4 - x_1 x_4),
$$
and the prime ideals in its primary decomposition are $\fp = (x_1, x_2)$ and $(x_1, x_2, x_4)$, hence $\codim(\Q) = 2$. The vector $\nu = (0, 0, -x_4, x_3)^T$ is a syzygy of degree $1$ for $\Theta$.

From the associativity formula, we may compute 
$$
e_1(\Q) = \length(\Q_{\fp}) \cdot \deg(R/\fp).
$$
Since $R/\fp \simeq k[x_3, \ldots, x_n]$, $\deg(R/\fp) = 1$. Localizing the sequence defining $\Q$ at the prime ideal $\fp$, we obtain
$$
R_{\fp}^4 \xrightarrow{\Theta_{\fp}} R_{\fp}^2 \to \Q_{\fp} \to 0,
$$
where we may write
$$
\Theta_{\fp} = \begin{bmatrix}
x_1 & x_2 & x_3 & x_4\\
(\lambda - \mu)x_1 & x_1 + (\lambda - \mu)x_2 & 0 & 0
\end{bmatrix}.
$$
We note that $\Theta_{\fp}$ sends the scaled fourth basis vector $u e_4$ to $(1,0)$ in $R_{\fp}^2 = (1,0) \cdot R_{\fp} \oplus (0,1) \cdot R_{\fp}$, where $u = 1/x_4 \in R_{\fp}$. Thus, we may compute the cokernel of $\Theta_{\fp}$ by considering the second row of the remaining matrix:
$$
\cok(\Theta_{\fp}) \simeq R_{\fp}/((\lambda- \mu)x_1, x_1 + (\lambda - \mu)x_2) \simeq R_{\fp}/(x_1, x_2) = R_{\fp}/\fp_{\fp} \simeq k,
$$
since $\lambda \neq \mu$. Therefore, $\length(\Q_{\fp}) = 1$ and $e_1(\Q) = 1$.

Let $\Theta = J_2(\lambda) \mid J_1(\lambda) \mid J_1(\mu) = \begin{bmatrix}
x_1 & x_2 & x_3 & x_4\\
\lambda x_1 & x_1+\lambda x_2 & \lambda x_3 & \mu x_4
\end{bmatrix}$, where $\lambda \neq \mu$. Here, the ideal of minors is given by
$$
I_2(\Theta) = (x_1^2, (\mu - \lambda)x_1 x_4, - x_1 x_3, (\mu - \lambda)x_2 x_4 - x_1 x_4, (\mu - \lambda)x_3 x_4),
$$
where the primes in the primary decomposition are $\fp = (x_1, x_4)$, $(x_1, x_2, x_3)$ and $(x_1, x_3, x_4)$, hence $\codim(\Q) = 2$. The vector $\nu = (-x_3, 0, x_1, 0)^T$ is a syzygy of degree $1$ for $\Theta$.

From the associativity formula, we may compute 
$$
e_1(\Q) = \length(\Q_{\fp}) \cdot \deg(R/\fp).
$$
Note that $\deg(R/\fp) = 1$ since $R/\fp \simeq k[x_2, x_3, x_5, \ldots, x_n]$. On the other hand, localizing the sequence of $\Q$ at $\fp$
$$
R_{\fp}^4 \xrightarrow{\Theta_{\fp}} R_{\fp}^2 \to \Q_{\fp} \to 0
$$
we may write the matrix as
$$
\Theta_{\fp} = \begin{bmatrix}
x_1 & x_2 & x_3 & x_4\\
0   & x_1 & 0   & (\mu - \lambda)x_4
\end{bmatrix},
$$
and seeing that $\Theta_{\fp}$ sends the scaled third basis vector $u e_3$ to $(1,0) \in R_{\fp}^2$, with $u = 1/x_3 \in R_{\fp}$, we may compute the cokernel from the second row in the remaining matrix:
$$
\cok(\Theta_{\fp}) \simeq R_{\fp}/(x_1, (\mu-\lambda)x_4) \simeq R_{\fp}/\fp_{\fp} \simeq k,
$$
since $\mu \neq \lambda$, and thus $\length(\Q_{\fp}) = 1$ and $e_1(\Q) = 1$.

Let $\Theta = J_1(\lambda) \mid J_1(\lambda) \mid J_1(\mu) \mid J_1(\rho) = \begin{bmatrix}
x_1 & x_2 & x_3 & x_4\\
\lambda x_1 & \lambda x_2 & \mu x_3 & \rho x_4
\end{bmatrix}$, where $\lambda \neq \mu, \rho \neq \lambda, \mu$. In this case, the ideal of minors is given by
$$
I_2(\Theta) = ((\mu-\lambda)x_1 x_3, (\mu-\lambda)x_2 x_3, (\rho - \lambda)x_1 x_4, (\rho - \lambda)x_2 x_4, (\rho-\mu)x_2 x_4),
$$
with prime ideals in its primary decomposition $\fp = (x_3, x_4)$, $(x_1, x_2, x_4)$ and $(x_1, x_2, x_3)$. The vector $\nu = (-x_2, x_1, 0, 0)^T$ is a syzygy of degree $1$ for $\Theta$.

From the associativity formula, we may compute 
$$
e_1(\Q) = \length(\Q_{\fp}) \cdot \deg(R/\fp).
$$
Since $R/\fp \simeq k[x_1,x_2, x_5, \ldots, x_n]$, we obtain $\deg(R/\fp) = 1$. Localizing the sequence of $\Q$ at $\fp$, we get
$$
R_{\fp}^4 \xrightarrow{\Theta_{\fp}} R_{\fp}^2 \to \Q_{\fp} \to 0
$$
where we may write
$$
\Theta_{\fp} = \begin{bmatrix}
x_1 & x_2 & x_3 & x_4\\
0   & 0   & (\mu - \lambda)x_3 & (\rho-\lambda)x_4
\end{bmatrix}.
$$
Since the matrix $\Theta_{\fp}$ sends the scaled first basis vector $u e_4$ to $(1,0)$ in $R_{\fp}^2$, with $u = 1/x_1 \in R_{\fp}$, we may compute the cokernel by looking at the second row of the remaining matrix:
$$
\cok(\Theta_{\fp}) \simeq R_{\fp}/((\mu-\lambda)x_3, (\rho - \lambda)x_4) = R_{\fp}/(x_3, x_4) =R_{\fp}/\fp_{\fp} \simeq k,
$$
since $\lambda \neq \mu, \rho$, and thus $\length(\Q_{\fp}) = 1$ and $e_1(\Q) = 1$.
\end{proof}

\begin{Proposition}\label[Proposition]{KWform-nearlyfree-codimQ-1}
Let the Kroenecker-Weierstrass form of a $2 \times 4$ linear matrix $\Theta$ be of the form
$
D_1 \mid J_1(\lambda) \mid J_1(\lambda)$ for $\lambda \neq 0$ or of the form $ J_1(\lambda) \mid J_1(\lambda) \mid J_1(\lambda) \mid J_1(\mu)$, where $\mu \neq \lambda$. Then $\Theta$ is nearly free, with $\codim(\Q) = 1$, initial degree $e = 1$ and Hilbert coefficients $e_0(\Q) = 1$, $e_1(\Q) = -1$. Moreover, the minimal free resolution for $\Theta$ is of the form
$$
0 \to R(-2) \to R^3(-1) \to R^4 \xrightarrow{\Theta} R^2(1).
$$
\end{Proposition}

\begin{proof}
If we assume that $\Theta$ is a linear matrix so that $e = 1$, $e_0(\Q) = 1$ and $e_1(\Q) = -1$, it follows from the Bourbaki degree formula that
\begin{align*}
\Bour(\Theta) &= (e-d)(e+e_0(\Q)) + \ell_{\Q} + \fq_{\Theta} + e_1(\Q)\\
&= (1-2)(1+1) + 1 + 3 - 1\\
&= -2+3 = 1.
\end{align*}
Moreover, the minimal free resolution is of the form of the claim, coming from the minimal free resolution of two hyperplanes in $\mathbb{P}^{n-1}$, namely
$$
0 \to R(-2) \to R^2(-1) \to I_\nu \to 0,
$$
for a choice of syzygy $\nu$ of degree $1$, since $s = e-d+e_0(\Q) = 0$. Thus, it suffices to show that $e = 1$, $e_0(\Q) = 1$ and $e_1(\Q) = -1$ for each case.

Let $\Theta = D_1 \mid J_1(\lambda) \mid J_1(\lambda) = \begin{bmatrix}
x_1 & 0 & x_2 & x_3\\
0   & x_1 & \lambda x_2 & \lambda x_3
\end{bmatrix},$ with $\lambda \neq 0$. The vector $\nu = (0, 0, -x_3, x_2)^T$ is a linear syzygy for $\Theta$, hence $e = 1$. The ideal of minors is
$$
I_2(\Theta) = (x_1^2, \lambda x_1 x_2, \lambda x_1 x_3, - x_1 x_2, - x_1 x_3),
$$
and the prime ideals in its primary decomposition are $\fp = (x_1)$ and $(x_1, x_2, x_3)$. Hence, $\codim(\Q) = 1$, and we may use the associativity formula to obtain
$$
e_0(\Q) = \length(\Q_{\fp}) \cdot \deg(R/\fp).
$$
Since $R/\fp \simeq k[x_2, \ldots, x_n]$, $\deg(R/\fp) = 1$. Localizing the sequence of $\Q$ at $\fp$, we obtain
$$
R_{\fp}^4 \xrightarrow{\Theta_{\fp}} R_{\fp}^2 \to \Q_{\fp} \to 0,
$$
where we may write
$$
\Theta_{\fp} = \begin{bmatrix}
x_1            & 0 & x_2 & x_3\\
-\lambda x_1   & x_1 & 0 & 0
\end{bmatrix}.
$$
Note that $\Theta_{\fp}$ sends the scaled fourth basis vector $ue_4$ to $(1,0)$ in $R_{\fp}^2 = (1,0) \cdot R_{\fp} \oplus (0,1) \cdot R_{\fp}$, with $u = 1/x_3$, and thus we may compute the cokernel of $\Theta_{\fp}$ from the second row of the remaining matrix:
$$
\cok(\Theta_{\fp}) \simeq R_{\fp}/(x_1, -\lambda x_1) \simeq R_{\fp}/\fp_{\fp} \simeq k,
$$
thus concluding that $\length(\Q) = 1$, and therefore $e_0(\Q) = 1$. For $n = 4$, we obtain $e_1(\Q) = -1$ using \cite{M2}, and we are able to extend this for $n \geq 4$ using \Cref{Lem:ext-of-var}.

Let $\Theta = J_1(\lambda) \mid J_1(\lambda) \mid J_1(\lambda) \mid J_1(\mu) = \begin{bmatrix}
x_1 & x_2 & x_3 & x_4\\
\lambda x_1   & \lambda x_2 & \lambda x_3 & \mu x_4
\end{bmatrix},$ with $\lambda \neq 0$. The vector $\nu = (-x_2, x_1, 0, 0)^T$ is a linear syzygy for $\Theta$, hence $e = 1$. The ideal of minors is
$$
I_2(\Theta) = ((\mu- \lambda)x_1x_4,(\mu-\lambda)x_2x_4,  (\mu - \lambda)x_4 x_3),
$$
and the prime ideals in its primary decomposition are $\fp = (x_4)$ and $(x_1, x_2, x_3)$. Hence, $\codim(\Q) = 1$, and from the associativity formula
$$
e_0(\Q) = \length(\Q_{\fp}) \cdot \deg(R/\fp).
$$
Since $R/\fp \simeq k[x_1, x_2, x_5, \ldots, x_n]$, $\deg(R/\fp) = 1$. Localizing the sequence of $\Q$ at $\fp$, we obtain
$$
R_{\fp}^4 \xrightarrow{\Theta_{\fp}} R_{\fp}^2 \to \Q_{\fp} \to 0
$$
where we may write
$$
\Theta_{\fp} = \begin{bmatrix}
x_1          & x_2 & x_3 & x_4\\
0            & 0   & 0   & (\mu - \lambda) x_4
\end{bmatrix}.
$$
Note that the matrix $\Theta_{\fp}$ sends the scaled third basis vector $u e_3$ from $R_{\fp}^4$ to $(1,0)$ in $R_{\fp}^2 = (1,0) \cdot R_{\fp} \oplus (0,1) \cdot R_{\fp}$, with $u = 1/x_3$. Thus, we may compute the cokernel of $\Theta_{\fp}$ by considering the second row of the remaining matrix, giving
$$
\cok(\Theta_{\fp}) \simeq R_{\fp}/((\mu - \lambda)x_4) \simeq R_{\fp}/\fp_{\fp} \simeq k,
$$
since $\lambda \neq \mu$, and thus $\length(\Q_{\fp}) = 1$ and we obtain $e_0(\Q) = 1$. For $n = 4$, we obtain from \cite{M2} the data $e_1(\Q) = -1$, which we may extend for $n \geq 4$ using \Cref{Lem:ext-of-var}.
\end{proof}

The last case remaining of the Kronecker--Weierstrass classification is the following one. It is particularly interesting because it fills the gap of achievable Bourbaki degrees for Jacobian matrices of pencils of quadrics $\Bour(\sigma) = 2$, while also inducing a non-integrable distribution when $n=4$, as we will comment in \Cref{ex:KW-form-Bour2}. For completeness, we rewrite the classification in \cite[Theorem 6.1]{Faenzi2025} from our point of view:

\begin{Theorem}\label[Theorem]{thm:pencils-quadrics-FJV}
Let $\Theta$ be a Jacobian matrix of a pair of quadrics $(f,g)$ with $n = 4$. Then, either:
\begin{itemize}
    \item[(a)] $\syz(\Theta) \simeq R(-1)^2$;
    \item[(b)] $\syz(\Theta) \simeq R \oplus R(s)$, when the initial degree $e = 0$;
    \item[(c)] When $\Theta$ is not free, there are two possibilities for minimal free resolutions, namely
    $$
    0 \to R(-3)^2 \to R(-2)^4 \to \syz(\Theta) \to 0,
    $$
    which is of Buchsbaum--Rim form with $\Bour(\Theta) = 3$, or
    $$
    0 \to R(-3) \to R(-2)^2 \oplus R(-1) \to \syz(\Theta) \to 0,
    $$
    of nearly free form, with $\Bour(\Theta) = 1$.
\end{itemize}
Moreover, $\Theta$ is locally free if and only if $\Theta$ is free.
\end{Theorem}

\begin{Theorem}\label[Theorem]{thm:KWform-Bour2}
Let the Kroenecker-Weierstrass form of a $2\times 4$ linear matrix $\Theta$ be of the form
$$
\Theta = D_2 \mid B_1 = \begin{bmatrix}
x_1 & x_2 & 0   & x_3\\
0   & x_1 & x_2 & x_4
\end{bmatrix}.
$$
Then $e_0(\Q) = 0$, $e_1(\Q) = 1$ and $e = 2$, which gives $\Bour(\Theta) = 2$. Moreover, the minimal free resolution for $\Theta$ is of the following form:
$$
0 \to R(-4) \to R^4(-3) \to R^5(-2) \to R^4 \xrightarrow{\Theta} R^2(1)\to \Q\to 0.
$$
\end{Theorem}

\begin{proof}
We have shown that the minimal degree for a syzygy is $e = 2$ in \Cref{D2B1}.

The ideal of minors of $\Theta$ is
$$
I_2(\Theta) = (x_1^2, x_1 x_2, x_2^2, x_1 x_4, x_2 x_4 - x_1 x_3, - x_2 x_3),
$$
with primes in the primary decomposition $\fp = (x_1, x_2)$ and $(x_1, x_2, x_3, x_4)$, so $\codim(\Q) = 2$ and $e_0(\Q) = 0$. To compute $e_1(\Q)$, we use the associativity formula with respect to $\fp$. Since $R/\fp \simeq k[x_3, \ldots, x_n]$, $\deg(R/\fp) = 1$. At the localization at $\fp$, we use the element $u = x_4/x_3 \in R_{\fp}$ to rewrite $\Theta_{\fp}$ as
$$
\Theta_{\fp} = \begin{bmatrix}
x_1 & x_2 & 0 & 1\\
-u x_1 & x_1 - u x_2 & x_2 & 0
\end{bmatrix}.
$$
Thus, the map $\Theta_{\fp}: R_{\fp}^4 \to R_{\fp}^2$ sends the fourth basis vector $e_4$ to $(1,0) \in R_{\fp}^2$. The image of $\Theta_{\fp}$ contains the direct summand $R_{\fp} \cdot (1,0) \subset R_{\fp}^2$, and we may compute the cokernel using the second row of the remaining matrix, as follows:
\begin{align*}
\cok(\Theta_{\fp}) \simeq R_{\fp}/( -u x_1, x_2, x_1 - u x_2 ) \simeq R_{\fp}/(x_1, x_2) = R_{\fp}/\fp_{\fp} \simeq k,
\end{align*}
so $\length(\Q_{\fp}) = 1$, and $e_1(\Q)=1$. Since $e = 2$, $e_0(\Q) = 0$ and $e_1(\Q) = 1$, it follows from the Bourbaki degree formula that $\Bour(\Theta) = 2$.

For $n = 4$, the Hilbert polynomial of $\Q$ is of the form $t + 3$, indicating that the genus of the projective curve associated to the ideal $I_\nu$ is a curve with degree $2$ and genus $g = -1$, from the exact sequences:
$$
0 \to \syz(\Theta) \to R^4 \xrightarrow{\Theta} R^2(1) \to \Q \to 0 \ \ , \ \ 0 \to R(-2) \to \syz(\Theta) \to I_\nu \to 0,
$$
and therefore it is a union of two skew lines. Using \Cref{Thm:Bourbaki-resolutions}, we conclude that the minimal free resolution for $\Theta$ in this case is of the form:
$$
0 \to R(-4) \to R^4(-3) \to R^5(-2) \to R^4 \xrightarrow{\Theta} R^2(1),
$$
coming from a lift of $I_\nu$. For $n \geq 5$, we have the same minimal free resolution, using \Cref{Lem:ext-of-var}.
\end{proof}


\section{The geometric point of view}\label{sec:geometrical-point-of-view}

In this section, we use the theory of sheaves and distributions on projective spaces to obtain some results for $n = 4$ and for the general case.

For $n=4$, we are generalizing the setting considered in \cite{Faenzi2025}, \cite{monteiro2025}, for logarithmic sheaves on $\mathbb{P}^3$, where authors assume the matrices $\Theta$ are \emph{Jacobian matrices}, that is, there are some homogeneous polynomials $f,g$ such that
$$
\Theta = \nabla (f,g) = \begin{bmatrix}
\nabla f\\
\nabla g
\end{bmatrix}
$$
In \cite{Faenzi2025}, it is shown that every logarithmic sheaf induces a codimension one foliation in $\mathbb{P}^3$ or, in other terms, the syzygy module becomes a submodule $\syz(\Theta)(1) \hookrightarrow T$, where $T$ is the graded module associated to the tangent sheaf of $\mathbb{P}^3$. Moreover, for it to define a \emph{distribution}, we assume the cokernel of this submodule is torsion-free, and for it to be a \emph{foliation}, we need an extra integrability condition.

\subsection{Syzygy modules and distributions}\label{subsec:syz-dist}

Without the assumption of being a Jacobian matrix, we may lose the condition of being a distribution. The main result of this section displays a sufficient condition for this to occur.

\begin{Theorem}\label[Theorem]{lem:existence-dist}
Let $R=k[x_1,\ldots,x_4]$ be a polynomial ring over $k$ and $\Theta$ be a $2\times 4$ matrix of rank $2$ in $R$ whose first and second rows consist of homogeneous polynomials of degrees $d_1$ and $d_2$, respectively. Let
$$
\varepsilon \doteq \begin{bmatrix}
x_1\\
x_2\\
x_3\\
x_4
\end{bmatrix}: R(-1) \to R^4
$$
denote the \emph{Euler vector}. The composition $\Theta \circ \varepsilon: R(-1) \to R(d_1) \oplus R(d_2)$ can be written in terms of two homogeneous polynomials, say
$$
\Theta \circ \varepsilon = \begin{bmatrix}
h_1\\
h_2
\end{bmatrix}.
$$
If $(h_1, h_2)$ is a regular sequence, then $\syz(\Theta)(1)$ is the graded module associated to the tangent sheaf of a codimension one distribution on $\mathbb{P}^3$. In particular, this always holds if $\Theta$ is a Jacobian matrix of a regular sequence of homogeneous polynomials $(f,g)$.
\end{Theorem}

\begin{proof}
From this assumption, we conclude that the cokernel module $E = \cok(\Theta \circ \varepsilon)$ is torsion-free, since it fails to be locally free precisely over the ideal $(h_1, h_2)$. Now, consider the commutative diagram of modules formed using the Euler exact sequence, which defines the tangent module $T$:
\[\begin{tikzcd}
	&& {R(-1)} & {R(-1)} \\
	0 & {\syz(\Theta)} & {R^4} & {\text{im } \Theta} & 0 \\
	0 & {\syz(\Theta)} & {T(-1)} & F & 0
	\arrow[equals, from=1-3, to=1-4]
	\arrow[from=1-3, to=2-3]
	\arrow[from=1-4, to=2-4]
	\arrow[from=2-1, to=2-2]
	\arrow[from=2-2, to=2-3]
	\arrow[equals, from=2-2, to=3-2]
	\arrow[from=2-3, to=2-4]
	\arrow[two heads, from=2-3, to=3-3]
	\arrow[from=2-4, to=2-5]
	\arrow[two heads, from=2-4, to=3-4]
	\arrow[from=3-1, to=3-2]
	\arrow[from=3-2, to=3-3]
	\arrow[from=3-3, to=3-4]
	\arrow[from=3-4, to=3-5]
\end{tikzcd}\]
We note that $F \doteq \cok(\syz(\Theta) \to T(-1)) \hookrightarrow E$ is a submodule, and therefore it is also torsion-free. But this condition is enough to say that the short exact sequence of sheaves associated to the bottom row defines a codimension one distribution on $\mathbb{P}^3$. 

If the matrix is Jacobian, note that, from the Euler relation, if $f$, $g$ are homogeneous of degrees $d_1+1, d_2+2$, we have
$$
\sum_{i=1}^4 x_i \partial_{i} f = (d_1+1) f, \ \sum_{i=1}^4 x_i \partial_{i} g = (d_2+1) g, 
$$
and hence $h_1 = (d_1+1)f$, $h_2 = (d_2 + 1)g$, so the claim follows.
\end{proof}

If a matrix $\Theta$ satisfies the conditions of the Lemma above, we say that $\syz(\Theta)$ \emph{induces a distribution}, and the associated distribution will be said to have a \emph{matrix presentation} if it arises from a diagram as above. All linear matrices, studied in \Cref{sec:linear-matrices}, satisfy the condition above. For an example of a distribution arising this way which is non-integrable, we turn to the case described in \Cref{thm:KWform-Bour2}.

\begin{Example}\label[Example]{ex:KW-form-Bour2}
Assume that the matrix $\Theta$ is written in the form below
$$
\Theta = \begin{bmatrix}
x_1 & x_2 & 0   & x_3\\
0   & x_1 & x_2 & x_4
\end{bmatrix}.
$$
Then, the polynomials $h_1, h_2$ considered in the claim of the Lemma are
\begin{align*}
h_1 &= x_1^2 + x_2^2 + x_3 x_4\\
h_2 &= x_1 x_2 + x_2 x_3 + x_4^2
\end{align*}
which form a complete intersection, so it does induce a codimension one distribution on $\mathbb{P}^3$, of degree two.

We may see that $\Theta$ is not a Jacobian matrix for this choice of coordinates, from the first line: if $f \in R$ such that $\nabla f = (x_1, x_2, 0, x_3)$, then
$$
1 = \partial_3 (x_3) = \partial_3 \partial_4 f \neq \partial_4 \partial_3 f = \partial_4(0) = 0,
$$
a contradiction. 

In \cite[Theorem $6.1$]{Faenzi2025}, the authors classify all possible syzygy modules $\syz(\Theta)$ arising from Jacobian matrices associated to pairs of polynomials $(f,g)$ with $d_1 = d_2 = 1$, where $(f,g)$ is called a \emph{pencil of quadrics}. If we look at their classification, there are cases with $\Bour(\Theta) \in \{0, 1, 3\}$. Moreover, they are locally free on the punctured spectrum if and only if they are free. The example considered above is locally free but not free, and it satisfies $\Bour(\Theta) = 2$, a value of Bourbaki degree missing in this previous classification.

The sheaf associated to $\syz(\Theta)$ is a \emph{null-correlation bundle}, and distributions with this tangent bundle were studied extensively in \cite{calvo2020codimension}, where the authors show that they are not integrable. This suggests that integrability could be related to the matrix being Jacobian. We also remark that this is the first matrix presentation of this kind for these distributions.
\end{Example}

For an example of a matrix that does not induce a distribution in this way, consider
$$
\Theta = \begin{bmatrix}
x_3    & x_4     & 0 & 0\\
x_1x_3 & x_1x_4  & -x_1x_2 & -x_2^2
\end{bmatrix}.
$$
Here, the syzygy module is free, with $\syz(\Theta) \simeq R(-1)^2$. It fails the condition of the Lemma, as the elements
\begin{align*}
h_1 &= x_1 x_3 + x_2 x_4 \\
h_2 &= x_1^2 x_3 + x_1 x_2 x_4 - x_1 x_2 x_3 - x_2^2 x_4\\
&= (x_1 - x_2)(x_1 x_3 + x_2 x_4)
\end{align*}
do not form a regular sequence. However, there is a codimension one distribution with tangent sheaf $\mathcal{O}_{\mathbb{P}^3}^{\oplus 2}(-1)$. Our condition does not classify whether a syzygy module is a tangent sheaf for a codimension one distribution or not, but rather it states a sufficient condition for a matrix to induce it via the Euler derivation.

\begin{Remark}\label[Remark]{rmk:distributions-normal-matrices}
Let us use the same notation from \Cref{lem:existence-dist}. If a matrix $\Theta$ is such that $\Theta \circ \varepsilon = (h_1, h_2)$ is not a regular sequence, we obtain $e_0(\Theta) \neq 0$. Indeed, let $f \in R$ be a homogeneous polynomial with $\deg(f) = t > 0$ such that 
$$
h_1 = f g_1, h_2 = f g_2,
$$
for some $g_1, g_2$ homogeneous. Now, let $M \doteq \text{coker} \Theta$. Then $\Fitt_0(M) = I_2(\Theta)$, and Fitting ideals satisfy a base-change formula, given by
$$
\Fitt_0(M \otimes_R S) = \Fitt_0(M) \cdot S
$$
whenever $R \to S$ is a ring map. Consider the particular case where $S = R/I$ and $I = (f)$, and consider the base change of the whole composition
$$
S(-1) \xrightarrow{\overline{\varepsilon}} S^4 \xrightarrow{\overline{\Theta}} S(d_1) \oplus S(d_2).
$$
Then the image of $\overline{\varepsilon}$ lies in $\ker(\overline{\Theta})$, since the image $(h_1, h_2) = 0 \mod I$. Thus the rank of $\overline{\Theta}=\Theta \otimes_R S$ is at most one, and therefore $\Fitt_0(M \otimes S) = (0)$. Coming back to the base-change formula, we get
$$
\Fitt_0(M) \cdot S/I = (0) \subset \frac{S}{I},
$$
and thus $\Fitt_0(M) \subset I$. But this gives $V(f) = V(I) \subset V(\Fitt_0(M)) = V(I_2(\Theta))$, hence $e_0(M) \geq t > 0$. In particular, if the ideal $I_2(\Theta)$ has codimension at least two, then $\Theta$ always induces a codimension one distribution.
\end{Remark}

\subsection{The case $n = 4$ with initial degree $e = 0, 1$.}

If we consider a minimal generating syzygy $\nu: R(-e) \to \syz(\Theta)$, we obtain another sequence
\[\begin{tikzcd}
	& {R(-e)} & {R(-e)} \\
	0 & {\syz(\Theta)} & {T(-1)} & {I_Z(t)} & 0 \\
	0 & {I_\nu(s)} & F & {I_Z(t)} & 0
	\arrow[equals, from=1-2, to=1-3]
	\arrow[from=1-2, to=2-2]
	\arrow[from=1-3, to=2-3]
	\arrow[from=2-1, to=2-2]
	\arrow[from=2-2, to=2-3]
	\arrow[from=2-2, to=3-2]
	\arrow[from=2-3, to=2-4]
	\arrow[from=2-3, to=3-3]
	\arrow[from=2-4, to=2-5]
	\arrow[from=2-4, to=3-4]
	\arrow[from=3-1, to=3-2]
	\arrow[from=3-2, to=3-3]
	\arrow[from=3-3, to=3-4]
	\arrow[from=3-4, to=3-5]
\end{tikzcd}\]
which relates the conormal module $F$ of a foliation by curves on $\mathbb{P}^3$ of degree $e$ at the middle column to the Bourbaki ideal $I_\nu(s)$ (see, for example, \cite{correa2023classification}). The number $e\geq 0$ is called the \emph{degree} of the foliation by curves. For degrees $e \in \{0, 1\}$ these are classified in \cite[Theorem 4]{GALEANO2022106840}. Using this classification, we follow the strategy described in \cite[Section 3]{monteiro2025} to obtain the results of this section. 

We change to sheaf notation, denoting by $\mathcal{F}=\Tilde{F}$ the associated sheaf to $F$ on $\mathbb{P}^3$ and by $\mathcal{I}_\nu$ the ideal sheaf associated to $I_\nu$.
We start by dualizing the associated short exact sequence at the bottom of the diagram above, namely
$$ 0 \to \mathcal{I}_\nu(s) \to \mathcal{F} \to \mathcal{I}_Z(t) \to 0 $$
to obtain a long exact sequence of sheaves, which ends at
\[\begin{tikzcd}
	\ldots & {\mathcal{E}xt^1(\mathcal{F}, \mathcal{O}_{\mathbb{P}^3})} & {\mathcal{E}xt^1(\mathcal{I}_\nu(s), \mathcal{O}_{\mathbb{P}^3})} & {\mathcal{E}xt^2(\mathcal{I}_Z(t), \mathcal{O}_{\mathbb{P}^3})} \\
	\ldots & {\mathcal{O}_W} & {\omega_\nu(4-s)} & {\mathcal{E}xt^3(\mathcal{O}_Z(t), \mathcal{O}_{\mathbb{P}^3})}
	\arrow[from=1-1, to=1-2]
	\arrow[from=1-2, to=1-3]
	\arrow[equals, from=1-2, to=2-2]
	\arrow[two heads, from=1-3, to=1-4]
	\arrow[equals, from=1-3, to=2-3]
	\arrow[equals, from=1-4, to=2-4]
	\arrow[from=2-1, to=2-2]
	\arrow[from=2-2, to=2-3]
	\arrow[two heads, from=2-3, to=2-4]
\end{tikzcd}\]
where $W$ is the singular scheme of the rank one distribution of degree one defined by $\mathcal{F}$ and $\omega_\nu \simeq \mathcal{E}xt^2(\mathcal{I}_\nu, \omega_{\mathbb{P}^3})$ is the dualizing sheaf of the curve defined by $I_\nu$. Then, since $\mathcal{E}xt^3(\mathcal{O}_Z(t), \mathcal{O}_{\mathbb{P}^3})$ is zero-dimensional, it follows that $\deg(R/I_\nu) \leq \deg(W)$, where all the contribution of codimension two of $\omega_\nu$ comes from the sheaf $\mathcal{O}_W$. 

The following application gives another proof of \Cref{prop:compressibility} in the particular case when $\syz(\Theta)$ induces a distribution.

\begin{Proposition}\label[Proposition]{prop:e-0-foliations}
Let $R=k[x_1,\ldots,x_4]$ be a polynomial ring over $k$ and $\Theta$ be a $2\times 4$ matrix of rank $2$ in $R$ whose first and second rows consist of homogeneous polynomials of degrees $d_1$ and $d_2$, respectively, such that $\syz(\Theta)$ induces a distribution. If the minimal degree of a syzygy of $\Theta$ is $e = 0$, then $\syz(\Theta)$ is free with $\syz(\Theta) \simeq R\oplus R(s)$, with $s = e_0(\Q)-d$.
\end{Proposition}

\begin{proof}
The foliation by curves on $\mathbb{P}^3$ described previously will be of degree zero, and \cite[Remark $4.4$]{correa2023classification} shows that their singular scheme consists of a single point. Therefore, it follows that $\deg(R/I_\nu) = 0$, and therefore $\Bour(\Theta) = 0$ with $\Theta$ being free. Moreover, from the short exact sequence
$$
0 \to R \to \syz(\Theta) \to R(s) \to 0
$$
we obtain the claimed splitting of $\syz(\Theta)$.
\end{proof}

\begin{Proposition}\label[Proposition]{prop:initial-degree-1-n-4}
Let $R=k[x_1,\ldots,x_4]$ be a polynomial ring over $k$ and $\Theta$ be a $2\times 4$ matrix of rank $2$ in $R$ whose first and second rows consist of homogeneous polynomials of degrees $d_1$ and $d_2$, respectively, such that $\syz(\Theta)$ induces a distribution. If the minimal degree of a syzygy of $\Theta$ is $e = 1$, then $\Bour(\Theta) \leq 2$.
\end{Proposition}
\begin{proof}
 From the classification at \cite[Theorem $4$]{GALEANO2022106840}, $\deg(W) \in \{0, 1, 2\}$, and thus $\deg(R/I_\nu) \leq 2$ from the previous considerations.
\end{proof}

We were not able to find examples of matrices $\Theta$ showing that the bound above is sharp, in the sense that $e = 1$ and $\Bour(\Theta) = 2$. As we have mentioned before, this relates to \Cref{quest:Bour-e^2-bound}.
        
\subsection{Semistability of sheaves and bounds for Hilbert coefficient}\label{subsec:semistability-bounds}

In this section, we use simple results from the $\mu$-semistability of sheaves (see \cite{huybrechts2010geometry}) on projective spaces to get a different proof of \Cref{Maximale0}.

Consider the sequence of sheaves on $\mathbb{P}^{n-1}$ induced by a $2\times4$ matrix $\Theta$:
$$
0 \to S \to \mathcal{O}_{\mathbb{P}^{n-1}}^{\oplus 4} \xrightarrow{\Theta} \mathcal{O}_{\mathbb{P}^{n-1}}(d_1) \oplus \mathcal{O}_{\mathbb{P}^{n-1}}(d_2) \to \mathcal{Q} \to 0
$$
where $S = \widetilde{\syz(\Theta)}$ is the sheaf associated to the graded $R$-module $\syz(\Theta)$. In this language, we note that $c_1(\Q) = e_0(\Q)$.

\begin{proof}[Alternative proof for \Cref{Maximale0}]
First, by the additivity of the first Chern class on the exact sequence above, $
c_1(S) = -d + e_0(\Q)$.
Since $\mathcal{O}_{\mathbb{P}^{n-1}}^{\oplus 4}$ is a $\mu$-semistable sheaf with slope $\mu=0$, we obtain an inequality:
$$
\mu(S) \leq \mu(\mathcal{O}_{\mathbb{P}^{n-1}}^{\oplus 4}) \Rightarrow -d+e_0(\Q) \leq 0 \Rightarrow e_0(\Q) \leq d.
$$
If, moreover, $e_0(\Q) = d$, then $\mu(S) = \mu(\mathcal{O}_{\mathbb{P}^{n-1}}^{\oplus 4})$, meaning that $S$ is also a $\mu$-semistable sheaf, and the Jordan-Holder blocks of $S$ must appear as Jordan-Holder blocks from the sheaf $\mathcal{O}_{\mathbb{P}^{n-1}}^{\oplus 4}$. This forces an isomorphism of sheaves $S \simeq \mathcal{O}_{\mathbb{P}^{n-1}}^{\oplus 2}$, showing the second claim.
\end{proof}

\subsection{Nearly free matrices and locally free sheaves}

The following proposition may be thought of as a generalization of the analogous result for Jacobian matrices of normal sequences obtained in \cite[Proposition 19]{monteiro2025}.
\begin{Proposition}
Let $R=k[x_1,\ldots,x_4]$ be a polynomial ring over $k$ and $\Theta$ be a $2\times 4$ matrix of rank $2$ in $R$ whose first and second rows consist of homogeneous polynomials of degrees $d_1$ and $d_2$, respectively. If $\Theta$ is nearly free, then $\syz(\Theta)$ is not locally free at the punctured spectrum.
\end{Proposition}

\begin{proof}
If $\Theta$ is nearly free, then $\deg(R/I_\nu) = 1$ and $B = L \subset \mathbb{P}^3$ is a projective line. Then, there is a short exact sequence
$$
0 \to \mathcal{O}_{\mathbb{P}^3}(-e) \to S \to \mathcal{I}_L(s) \to 0,
$$
where $s=e-d+e_0(\Q)$. This sequence corresponds to an extension class inside the group $\Ext^1(\mathcal{I}_L(s), \mathcal{O}_{\mathbb{P}^3}(-e))$. Using \cite[Proposition 6.7]{hartshorne2013algebraic} and Serre dualities on $\mathbb{P}^3$ and on $L \simeq \mathbb{P}^1$, we note there are isomorphisms
\begin{align*}
\Ext^1(\mathcal{I}_L(s), \mathcal{O}_{\mathbb{P}^3}(-e)) &\simeq \Ext^1(\mathcal{I}_L, \mathcal{O}_{\mathbb{P}^3}(-e-s))\\
&\simeq \Ext^2(\mathcal{O}_L, \mathcal{O}_{\mathbb{P}^3}(-e-s))\\
&\simeq \Ext^1(\mathcal{O}_{\mathbb{P}^3}(-e-s), \mathcal{O}_L(-4))^{*}\\
&\simeq H^1(\mathcal{O}_L(-4+e+s))^{*}\\
&\simeq H^0(\mathcal{O}_L(2-e-s))^{*} = H^0(\mathcal{O}_L(2+d-e_0(\Q)))^{*}.\\
\end{align*}
The locally free extensions correspond to nowhere vanishing sections on $H^0(\mathcal{O}_L(2+d-e_0(\Q)))$. But since $e_0(\Q) \leq d$, $2+d-e_0(\Q) \geq 2 > 0$, and thus there are no such sections inside this vector space. Thus, we conclude that the extensions are never locally free.
\end{proof}

\end{document}